\documentclass{amsart}
\usepackage{amsmath,amsfonts,amssymb,amscd,amsthm,epsf,graphicx,hyperref}
\font\cuf=cmtt8
\newcommand{\curl}[1]{{\cuf #1}}
\newtheorem{prop}{Proposition}[section]
\newtheorem{thm}[prop]{Theorem}
\newtheorem{cor}[prop]{Corollary}
\newtheorem{ques}[prop]{Question}
\newtheorem{lem}[prop]{Lemma}

\theoremstyle{definition}
\newtheorem{de}[prop]{Definition}
\newtheorem{example}[prop]{Example}
\newtheorem{examples}[prop]{Examples}
\theoremstyle{remark}
\newtheorem*{remark}{Remark}            

\def\C{{\mathbb C}}

\def\Z{{\mathbb Z}}
\def\Q{{\mathbb Q}}
\def\R{{\mathbb R}}
\def\L{{\mathbb L}}
\def\inter{\mathop{\rm Int}\nolimits}

\def\im{\mathop{\rm Im}\nolimits}

\def\id{\mathop{\rm id}\nolimits}

\def\co{\colon\thinspace}
\def\emb{{\hookrightarrow}}
\def\E{\mathcal E}
\def\limin{\mathop{\lim_\leftarrow}\nolimits}
\def\Ebad{\mathcal E_{\rm{bad}}}
\def\Eo{\mathcal E_{\rm{O}}}
\def\EML{\mathcal E_{\rm{ML}}}
\def\Sbad{ S_{\rm{bad}}}
\def\SML{S_{\rm{ML}}}
\def\gbad{\gamma_{\rm{bad}}}
\def\gML{\gamma_{\rm{ML}}}
\def\r{\mathcal R}
\def\M{\mathcal M}
\def\Sm{\mathcal S}

\def\Q{\mathcal Q}
\begin{document}
\title{On uniqueness of end sums and 1-handles at infinity}

\author[J.~Calcut]{Jack S. Calcut}
\address{Department of Mathematics\\
         Oberlin College\\
         Oberlin, OH 44074}
\email{jcalcut@oberlin.edu}
\urladdr{\href{http://www.oberlin.edu/faculty/jcalcut/}{\curl{http://www.oberlin.edu/faculty/jcalcut/}}}

\author[R.~Gompf]{Robert E. Gompf}
\address{The University of Texas at Austin\\
Mathematics Department RLM 8.100\\
Attn: Robert Gompf\\
2515 Speedway Stop C1200\\
Austin, Texas 78712-1202}
\email{gompf@math.utexas.edu}
\urladdr{\href{https://www.ma.utexas.edu/users/gompf/}{\curl{https://www.ma.utexas.edu/users/gompf/}}}

\begin{abstract}
For oriented manifolds of dimension at least 4 that are simply connected at infinity, it is known that end summing is a uniquely defined operation. Calcut and Haggerty showed that more complicated fundamental group behavior at infinity can lead to nonuniqueness. The present paper examines how and when uniqueness fails. Examples are given, in the categories \textsc{top}, \textsc{pl} and \textsc{diff}, of nonuniqueness that cannot be detected in a weaker category (including the homotopy category). In contrast, uniqueness is proved for Mittag-Leffler ends, and generalized to allow slides and cancellation of (possibly infinite) collections of 0- and 1-handles at infinity. Various applications are presented, including an analysis of how the monoid of smooth manifolds homeomorphic to $\R^4$ acts on the smoothings of any noncompact 4-manifold.
\end{abstract}
\maketitle


\section{Introduction}

Since the early days of topology, it has been useful to combine spaces by simple gluing operations. The connected sum operation for closed manifolds has roots in nineteenth century surface theory, and its cousin, the boundary sum of compact manifolds with boundary, is also classical. These two operations are well understood. In the oriented setting, for example, the connected sum of two connected manifolds is unique, as is the boundary sum of two manifolds with connected boundary. The boundary sum has an analogue for open manifolds, the \emph{end sum}, which has been used in various dimensions since the 1980s, but is less well known and understood. In contrast with boundary sums, end sums of one-ended oriented manifolds need not be uniquely determined, even up to proper homotopy \cite{CH}. The present paper explores uniqueness and its failure in more detail. To illustrate the subtlety of the issue, we present examples in various categories (homotopy, \textsc{top}, \textsc{pl}, and \textsc{diff}) where uniqueness fails, but the failure cannot be detected in weaker categories. In counterpoint, we find general hypotheses under which the operation {\em is} unique in all categories and apply this result to exotic smoothings of open 4-manifolds. Our results naturally belong in the broader context of {\em attaching handles at infinity}. We obtain general uniqueness results for attaching collections of 0- and 1-handles at infinity, generalizing handle sliding and cancellation. We conclude that end sums, and more generally, collections of handles at infinity with index at most one, can be controlled in broad circumstances, although deep questions remain.

End sums are the natural analogue of boundary sums. To construct the latter, we choose codimension zero embeddings of a disk into the boundaries of the two summands, then use these to attach a 1-handle. For an end sum of open manifolds, we attach a 1-handle at infinity, guided by a properly embedded ray in each summand. Informally, we can think of the 1-handle at infinity as a piece of tape joining the two manifolds; see Definition~\ref{onehandles} for details. Boundary summing two compact manifolds then has the effect of end summing their interiors. While this notion of end summing seems obvious, the authors have been unable to find explicit appearences of it before the second author's 1983 paper \cite{threeR4} and sequel \cite{infR4} on exotic smoothings of $\R^4$. However, the germ of the idea may be perceived in Mazur's 1959 paper \cite{M59} and Stallings' 1965 paper \cite{Stall65}. End summing  was used in \cite{infR4} to construct infinitely many exotic smoothings of $\R^4$. The appendix of that paper showed that the operation is well-defined in that context, so is independent of choice of rays and their order (even for infinite sums). Since then, the second author and others have continued to use end summing with an exotic $\R^4$  for constructing many exotic smoothings on various open 4-manifolds, e.g., Taylor (1997) \cite[Theorem~6.4]{T}, Gompf (2017) \cite[Section~7]{MinGen}. The operation has also been subsequently used in other dimensions, for example by Ancel (unpublished) in the 1980s to study high-dimensional Davis manifolds, and by Tinsley and Wright (1997) \cite{TW97} and Myers (1999) \cite{My} to study 3-manifolds. In 2012, the first author, with King and Siebenmann, gave a somewhat general treatment \cite{CKS} of end sum (called {\em CSI, connected sum at infinity}, therein) in all dimensions and categories (\textsc{top}, \textsc{pl}, and \textsc{diff}). One corollary gave a classification of multiple hyperplanes in $\R^n$ for all $n\neq3$, which was recently used by Belegradek \cite{B14} to study certain interesting open aspherical manifolds. Most recently, Sparks (2018) \cite{Sp} used infinite end sums to construct uncountably many contractible topological 4-manifolds obtained by gluing two copies of $\R^4$ along a subset homeomorphic to $\R^4$.

While \cite{infR4} showed that end sums are uniquely determined for oriented manifolds homeomorphic to $\R^4$, uniqueness fails in general for multiple reasons. The most obvious layer of difficulty already occurs for the simpler operation of boundary summing. In that case, when a summand has disconnected boundary, we must specify which boundary component to use. For example, nondiffeomorphic boundary components can lead to boundary sums with nondiffeomorphic boundaries. We must also be careful to specify orientations --- a pair of disk bundles over $S^2$ with nonzero Euler numbers can be boundary summed in two different ways, distinguished by their signatures (0 or $\pm 2$).  In general, we should specify an orientation on each orientable boundary component receiving a 1-handle. Similarly, for end sums and 1-handles at infinity, we must specify which ends of the summands we are using and an orientation on each such end (if orientable).

Unlike boundary sums, however, end sums have a more subtle layer of nonuniqueness. One difficulty is specific to dimension 3: the rays in use can be knotted. Myers \cite{My} showed that uncountably many homeomorphism types of contractible manifolds can be obtained by end summing two copies of $\R^3$ along knotted rays. For this reason, the present paper focuses on dimensions above 3. However, another difficulty persists in high dimensions: rays determining a given end need not be properly homotopic. The first author and Haggerty \cite{CH} constructed examples of pairs of one-ended oriented $n$-manifolds ($n\ge4$) that can be summed in different ways, yielding manifolds that are not even properly homotopy equivalent. We explore this phenomenon more deeply in Section~\ref{Nonunique}. After sketching the key example of \cite{CH} in Example~\ref{CH}, we exhibit more subtle examples of nonuniqueness of end summing (and related constructions) on fixed oriented ends. Examples~\ref{htpy} include topological 5-manifolds with properly homotopy equivalent but nonhomeomorphic end sums on the same pair of ends, and \textsc{pl} $n$-manifolds (for various $n\ge9$) whose end sums are properly homotopy equivalent but not \textsc{pl}-homeomorphic. Unlike other examples in this section, those in Examples~\ref{htpy} have extra ends or boundary components; the one-ended case seems more elusive. Examples~\ref{PL} provide end sums of smooth manifolds ($n\ge8$) that are \textsc{pl}-homeomorphic  but not diffeomeorphic. The analogous construction in dimension 4 gives smooth manifolds whose end sums are naturally identified in the topological category, but whose smoothings are not stably isotopic. Distinguishing their diffeomorphism types seems difficult.

These failures of uniqueness arise from complicated fundamental group behavior at the relevant ends, contrasting with uniqueness associated with the simply connected end of $\R^4$. Section~\ref{Unique} examines more generally when ends are simple enough to guarantee uniqueness of end sums and 1-handle attaching. In dimensions 4 and up, it suffices for the end to satisfy the {\em Mittag-Leffler} condition (also called {\em semistability}), whose definition we recall in Section~\ref{Unique}. Ends that are simply connected or topologically collared are Mittag-Leffler; in fact, the condition can only fail when the end requires infinitely many $(n-1)$-handles in any topological handle decomposition (Proposition~\ref{Morse}). For example, Stein manifolds of complex dimension at least 2 have (unique) Mittag-Leffler ends. (See Corollaries~\ref{stein} and \ref{stein1h}, and Theorem~\ref{JB} for an application to 4-manifold smoothing theory.) The Mittag-Leffler condition is necessary and sufficient to guarantee that any two rays approaching the end are properly homotopic. This fact traces back at least to Geoghegan in the 1980s, and appears to have been folklore since the preceding decade. (See also Edwards and Hastings \cite{EH76}, Mihalik \cite[Thm.~2.1]{Mi83}, and \cite{Ge08}). The first author and King worked out an algebraic classification of proper rays up to proper homotopy on an arbitrary end in 2002. This material was later excised from the 2012 published version of \cite{CKS} due to length considerations and since a similar proof had appeared in Geoghegan's text \cite{Ge08} in the mean time. The present paper gives a much simplified version of the proof, dealing only with the Mittag-Leffler case, in order to highlight the topology underlying the algebraic argument (Lemma~\ref{ML}). This lemma leads to a general statement (Theorem~\ref{main}) about attaching countable collections of 1-handles to an open manifold. The following theorem is a special case.

\begin{thm}\label{introMain} Let $X$ be a (possibly disconnected) $n$-manifold, $n\ge4$. Then the result of attaching a (possibly infinite) collection of 1-handles at infinity to some oriented  Mittag-Leffler ends of $X$ depends only on the pairs of ends to which each 1-handle is attached, and whether their orientations agree.
\end{thm}

\noindent Note that uniqueness of end sums along Mittag-Leffler ends (preserving orientations) is a special case. Theorem~\ref{main} also deals with ends that are nonorientable or not Mittag-Leffler.

Theorem~\ref{main} has consequences for open 4-manifold smoothing theory, which we explore in Section~\ref{Smooth}. The theorem easily implies the result from \cite{infR4} that the oriented diffeomorphism types of 4-manifolds homeomorphic to $\R^4$ form a monoid $\r$ under end sum, allowing infinite sums that are independent of order and grouping. This monoid acts on the set $\Sm(X)$ of smoothings (up to isotopy) of any given oriented 4-manifold $X$ with a Mittag-Leffler end, and more generally a product of copies of $\r$ acts on $\Sm(X)$ through any countable collection of Mittag-Leffler ends (see Corollary~\ref{R4}). One can also deal with arbitrary ends by keeping track of a family of proper homotopy classes of rays. Similarly, one can act on $\Sm(X)$ by summing with exotic smoothings of $S^3\times \R$ along properly embedded lines (Corollary~\ref{SxR}), or modify smoothings along properly embedded star-shaped graphs. While summing with a fixed exotic $\R^4$ is unique for an oriented (or nonorientable) Mittag-Leffler end, Section~\ref{Nonunique} suggests that there should be examples of nonuniqueness when the end of $X$ is not Mittag-Leffler. However, such examples seem elusive, prompting the following natural question.

\begin{ques}\label{R4sum} Let $X$ be a smooth, one-ended, oriented 4-manifold. Can summing $X$ with a fixed exotic $\R^4$, preserving orientation, yield different diffeomorphism types depending on the choice of ray in $X$?
\end{ques}

\noindent We show (Proposition~\ref{subtle_prop}) that such examples would be quite difficult to detect.

Having studied the uniqueness problem for adding $1$-handles at infinity, we progress in Section~\ref{slides} to uniqueness of adding collections of $0$- and $1$-handles at infinity (Theorem~\ref{maincancel}). It turns out that, when adding countably many handles of index $0$ and $1$, the noncompact case is simpler than for compact handle addition. As an application of Theorem~\ref{maincancel}, we present (Theorem~\ref{hut}) a very natural and partly novel proof of the hyperplane unknotting theorem of Cantrell~\cite{C63} and Stallings~\cite{Stall65}: each proper embedding of $\R^{n-1}$ in $\R^n$, $n\geq 4$, is unknotted (in each category \textsc{diff}, \textsc{pl}, and \textsc{top}). An immediate corollary is the \textsc{top} Schoenflies theorem: the closures of the two complementary regions of a (locally flat) embedding of $S^{n-1}$ in $S^n$, $n\geq4$, are topological disks. Mazur's infinite swindle still lies at the heart of our proof of the hyperplane unknotting theorem. The novelty in our proof consists of the supporting framework of $0$- and $1$-handle additions, slides, and cancellations at infinity.

Throughout the text, we take manifolds to be Hausdorff with countable basis, so with only countably many components. We allow boundary, and note that the theory is vacuous unless there is a noncompact component. {\em Open} manifolds are those with no boundary and no compact components. We work in a category \textsc{cat} that can be \textsc{diff}, \textsc{pl}, or \textsc{top}. For example, \textsc{diff} homeomorphisms are the same as diffeomorphisms. Embeddings (particularly with codimension zero) are not assumed to be proper. (Proper means the preimage of every compact set is compact.) In \textsc{pl} and \textsc{top}, embeddings are assumed to be locally flat (as is automatically true in \textsc{diff}). It follows that in each category, codimension-one two-sided embeddings in $\inter X$ are bicollared (Brown \cite{Brown} in \textsc{top}; see Connelly \cite{Co71} for a simpler proof in both  \textsc{top} and \textsc{pl}). Furthermore, a \textsc{cat} proper embedding $\gamma\co Y\emb X^n$ of a \textsc{cat} 1-manifold $Y$ with $b_1(Y)=0$ and $\gamma^{-1}(\partial X)=\emptyset$ extends to a \textsc{cat} proper embedding $\overline{\nu}\co Y\times D^{n-1}\emb X^n$ whose boundary (after rounding corners in \textsc{diff}) is bicollared. (This is easy in \textsc{diff} and \textsc{pl}, and follows in \textsc{top} by a classical argument: Cover suitably by charts exhibiting $Y$ as locally flat, then stretch one chart consecutively through the others.) If we radially identify $\R^{n-1}$ with $\inter D^{n-1}$, $\overline{\nu}$ determines an embedding $\nu\co Y\times\R^{n-1}\emb X$. We call $\nu$ and $\overline{\nu}$ {\em tubular neighborhood maps}, and their images open (resp.\ closed) {\em tubular neighborhoods} of $Y$. Thus, an open tubular neighborhood extends to a closed tubular neighborhood by definition.

\section{1-handles at infinity}\label{handles}

We begin with our procedure for attaching 1-handles at infinity.

\begin{de}\label{onehandles} A {\em multiray} in a \textsc{cat} $n$-manifold $X$ is a \textsc{cat} proper embedding $\gamma\co S\times [0,\infty)\emb X$, with $\gamma^{-1}(\partial X)=\emptyset$, for some discrete (so necessarily countable) set $S$ called the {\em index set} of $\gamma$. If the domain has a single component, $\gamma$ will be called a {\em ray}. Given two multirays $\gamma^-,\gamma^+\co S\times [0,\infty)\emb X$ with disjoint images, choose tubular neighborhood maps $\nu^\pm\co S\times [0,\infty)\times\R^{n-1}\emb X$ with disjoint images, and let $Z$ be the \textsc{cat} manifold obtained by gluing $S\times[0,1]\times\R^{n-1}$ to $X$ using identifications $\nu^\pm\circ(\id_S\times\varphi^\pm\times\rho^\pm)$, where $\varphi^-\co[0,\frac12)\to[0,\infty)$ and $\varphi^+\co(\frac12,1]\to[0,\infty)$ and $\rho^\pm\co\R^{n-1}\to\R^{n-1}$ are diffeomorphisms, with $\rho^\pm$ chosen so that $\varphi^\pm\times\rho^\pm$ preserves orientation. Then $Z$ is obtained by {\em attaching 1-handles at infinity} to $X$ along $\gamma^-$ and $\gamma^+$ (see Figure~\ref{fig:onehandle}).
\end{de}

\begin{figure}[htbp!]
    \centerline{\includegraphics[scale=1.0]{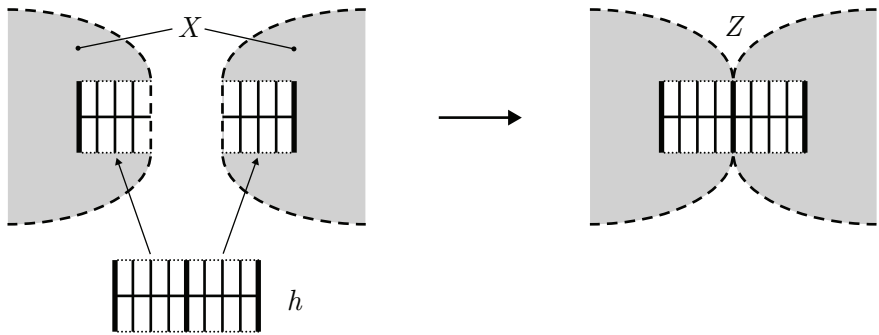}}
    \caption{Data for attaching $h$, a $1$-handle at infinity, to the $n$-manifold $X$ (left) and resulting $n$-manifold $Z$ (right).}
\label{fig:onehandle}
\end{figure}

\noindent The case of handle attaching where $S$ is a single point and $X$ has two components that are connected by the 1-handle at infinity is called the {\em end sum} or {\em connected sum at infinity} in the literature. In general, we will see that $Z$ depends in a subtle way on the choice of images of $\gamma^\pm$ (Section~\ref{Nonunique}), but not on the parametrizations of their rays. It depends on the orientations locally induced by $\nu^\pm$, but is otherwise independent of the choices of maps $\nu^\pm$, $\varphi^\pm$ and $\rho^\pm$. (Independence follows from the stronger Theorem~\ref{main} when $n\ge4$, and by a similar method in lower dimensions.) By reparametrizing the maps $\varphi^\pm$, we can change their domains to smaller neighborhoods of the endpoints of $[0,1]$ without changing $Z$, making it more obvious that attaching compact 1-handles to the boundary of a compact manifold has the effect of attaching handles at infinity to the interior. Yet another description of handle attaching at infinity is to remove the interiors of the closed tubular neighborhoods from $X$ and glue together the resulting $\R^{n-1}$ boundary components. Some articles (eg.\ \cite{CKS}, \cite{Sp}) use this perspective for defining end sums. It can be useful to start more generally, with any countable collection of disjoint rays, allowing clustering (for example, to preserve an infinite group action \cite{groupR4}). However, this gains  no actual generality, since we can transform such a collection to a multiray by suitably truncating the domains of the rays to achieve properness of the combined embedding.

\begin{remark} Handles at infinity of higher index are also useful \cite{MfdQuot}, although additional subtleties arise. For example, a Casson handle $CH$ can be attached to an unknot in the boundary of a 4-ball $B$ so that the interior of the resulting smooth 4-manifold is not diffeomorphic to the interior of any compact manifold. However, $\inter CH$ is diffeomorphic to $\R^4$, so we can interchange the roles of $\inter CH$ and $\inter B$, exhibiting the manifold as $\R^4$ with a 2-handle attached at infinity. The latter is attached along a properly embedded $S^1\times[0,\infty)$ in $\R^4$ that is topologically unknotted but smoothly knotted, and cannot be smoothly compactified to an annulus in the closed 4-ball. This proper annulus seems analogous to a knotted ray in a 3-manifold, but is more subtle since it is unknotted in \textsc{top}.
\end{remark}

Variations on the above 1-handle construction are used in \cite{MinGen}. Let $X$ be a topological 4-manifold with a fixed smooth structure, and let $R$ be an exotic $\R^4$ (a smooth, oriented manifold homeomorphic but not diffeomorphic to $\R^4$). Choose a smooth ray in $X$, and homeomorphically identify a smooth, closed tubular neighborhood $N$ of it with the complement of a tubular neighborhood of a ray in $R$. Transporting the smooth structure from $R$ to $N$, where it fits together with the original one on $X-\inter N$, we obtain a new smooth structure on $X$ diffeomorphic to an end sum of $X$ and $R$. The advantage of this description is that it fixes the underlying topological manifold, allowing us to assert, for example, that the two smooth structures are stably isotopic. Another variation from \cite{MinGen} is to sum a smooth structure with an exotic $\R\times S^3$ along a smooth, properly embedded line in each manifold, with one line topologically isotopic to $\R\times\{p\}\subset\R\times S^3$. (We order the factors this way instead of the more commonly used $S^3\times\R$ so that the obvious identification with $\R^4-\{0\}$ preserves orientation.) One can similarly change a smooth structure on a high-dimensional \textsc{pl} manifold by summing along a line with $\R\times\Sigma$ for some exotic sphere $\Sigma$. We exhibit these operations in Section~\ref{Smooth} as well-defined monoid actions on the set of isotopy classes of smoothings of a fixed topological manifold. One can also consider \textsc{cat} sums along lines in general. We discuss nonuniqueness of this latter operation in Section~\ref{Nonunique} as a prelude to discussing subtle end sums.

There are several obvious sources of nonuniqueness for attaching 1-handles at infinity. For attaching 1-handles in the compact setting, the result can depend both on orientations and on choices of boundary components. We will consider orientations in Section~\ref{Unique}, but now recall the noncompact analogue of the set of boundary components, the space of ends of a manifold (e.g., \cite{HR}). This only depends on the underlying \textsc{top} structure of a \textsc{cat} manifold $X$ (and generalizes to other  spaces). A {\em neighborhood of infinity} in $X$ is the complement of a compact set, and a {\em neighborhood system of infinity} is a nested sequence $\{U_i|i\in\Z^+\}$ of neighborhoods of infinity with empty intersection, and with the closure of $U_{i+1}$ contained in $U_i$ for all $i\in\Z^+$.

\begin{de} For a fixed neighborhood system $\{U_i\}$ of infinity, the {\em space of ends} of $X$ is given by $\E=\E (X)=\limin \pi_0(U_i)$.
\end{de}

\noindent That is, an end $\epsilon\in\E (X)$ is given by a sequence $V_1\supset V_2\supset V_3\supset\cdots$, where each $V_i$ is a component of $U_i$. For two different neighborhood systems of infinity for $X$, the resulting spaces $\E (X)$ can be canonically identified: The set is preserved when we pass to a subsequence, but any two neighborhood systems of infinity have interleaved subsequences. A {\em neighborhood} of the end $\epsilon$ is an open subset of $X$ containing one of the subsets $V_i$. This notion allows us to topologize the set $X\cup\E(X)$ so that $X$ is homeomorphically embedded as a dense open subset and $\E(X)$ is totally disconnected \cite{Fr}. (The new basis elements are the components of each $U_i$, augmented by the ends of which they are neighborhoods.) The resulting space is Hausdorff with a countable basis. If $X$ has only finitely many components, this space is compact, and called the {\em Freudenthal} or {\em end compactification} of $X$. In this case, $\E(X)$ is homeomorphic to a closed subset of a Cantor set.

Ends can also be be described using rays, most naturally if we allow the rays to be singular. We call a continuous, proper map $\gamma\co S\times[0,\infty)\to X$ ($S$ discrete and countable) a {\em singular multiray}, or a {\em singular ray} if $S$ is a single point. Every singular ray $\gamma$ in a manifold $X$ determines an end $\epsilon_\gamma\in\E(X)$. This is because $\gamma$ is proper, so every neighborhood $U$ of infinity in $X$ contains $\gamma([k,\infty))$ for sufficiently large $k$, and this image lies in a single component of $U$. In fact, an alternate definition of $\E(X)$ is as the set of equivalence classes of singular rays, where two such are considered equivalent if their restrictions to $\Z^+$ are properly homotopic. A singular multiray $\gamma\co S\times [0,\infty)\emb X$ then determines a function $\epsilon_\gamma\co S\to\E(X)$ that is preserved under proper homotopy of $\gamma$. Attaching 1-handles at infinity depends on these functions for $\gamma^-$ and $\gamma^+$, just as attaching compact 1-handles depends on choices of boundary components, with examples of the former easily obtained from the latter by removing boundary. We will find more subtle dependence on the defining multirays in the next section, but a weak condition preventing these subtleties in Section~\ref{Unique}.

\section{Nonuniqueness}\label{Nonunique}

We now investigate examples of nonuniqueness in the simplest setting. In each case, we begin with an open manifold $X$ with finitely many ends, and attach a single 1-handle at infinity, at a specified pair of ends. We assume the 1-handle respects a preassigned orientation on $X$. For attaching 1-handles in the compact setting, this would be enough information to uniquely specify the result, but we demonstrate that uniqueness can still fail for a 1-handle at infinity. It was shown in \cite{CH} that even the proper homotopy type need not be uniquely determined; Example~\ref{CH} below sketches the simplest construction from that paper. Our subsequent examples are more subtle, having the same proper homotopy (or even $\textsc{cat}'$ homeomorphism) type but distinguished by their \textsc{cat} homeomorphism types. 

All of our examples necessarily have complicated fundamental group behavior at infinity, since Section~\ref{Unique} proves uniqueness when the fundamental group is suitably controlled. We obtain the required complexity by the following construction, which generalizes examples of \cite{CH}:

\begin{de} For an oriented \textsc{cat} manifold $X$, let $\gamma^-,\gamma^+\co S\times [0,\infty)\emb X$ be multirays with disjoint images. {\em Ladder surgery} on $X$ along $\gamma^-$ and $\gamma^+$ is orientation-preserving surgery on the infinite family of 0-spheres given by $\{\gamma^-(s,n),\gamma^+(s,n)\}$ for each $s\in S$ and $n\in\Z^+$. That is, we find disjoint \textsc{cat} balls centered at the points $\gamma^\pm(s,n)$, remove the interiors of the balls, and glue each resulting pair of boundary spheres together by a reflection (so that the orientation of $X$ extends).
\end{de}

\noindent It is not hard to verify that the resulting oriented \textsc{cat} homeomorphism type only depends on the end functions $\epsilon_{\gamma^\pm}$ of the multirays; see Corollary~\ref{ladder} for details and a generalization to unoriented manifolds. If $X$ has two components $X_1$ and $X_2$, each with $k$ ends, any bijection from $\E(X_1)$ to $\E(X_2)$ determines a connected manifold with $k$ ends obtained by ladder surgery with $S=\E(X_1)$. Such a manifold will be called a {\em ladder sum} of $X_1$ and $X_2$. For closed, connected, oriented $(n-1)$-manifolds $M$ and $N$, we let $\L(M,N)$ denote the ladder sum of the two-ended $n$-manifolds $\R\times M$ and $\R\times N$, for the bijection preserving the ends of $\R$. (This is a slight departure from \cite{CH}, which used the one-ended manifold $[0,\infty)$ in place of $\R$.)  Note that any ladder surgery transforms its multirays $\gamma^\pm$ into infinite unions of circles, and surgery on all these circles (with any framings) results in the manifold obtained from $X$ by adding 1-handles at infinity along $\gamma^\pm$. (This is easily seen by interpreting the surgeries as attaching 1- and 2-handles to $I\times X$.)

The examples in~\cite{CH} are naturally presented in terms of ladder sums and attaching 1-handles at infinity. They represent the simplest type of example where a single 1-handle may be attached at infinity in essentially distinct ways, namely an orientation-preserving end sum of one-ended manifolds.

\begin{example}\label{CH} {\bf Homotopy inequivalent end sums (one-ended) \cite{CH}.}
For a fixed prime $p>1$, let $E$ denote the $\R^2$-bundle over $S^2$ with Euler number $-p$ (so $E$ has a neighborhood of infinity diffeomorphic to $\R\times L(p,1)$). Let $Y$ be the ladder sum of $E$ and $\R^4$. We will attach a single 1-handle at infinity to the disjoint union $X=Y \sqcup E$ in two ways to produce distinct, one-ended, boundaryless manifolds $Z_0$ and $Z_1$. Let $\gamma_0$ and $\gamma_1$ be rays in $Y$, with $\gamma_0$ lying in the $E$ summand and $\gamma_1$ lying in the $\R^4$ summand. Let $\gamma$ be any ray in $E$, and let $Z_i$ be obtained from $X$ by attaching a 1-handle at infinity along $\gamma_i$ and $\gamma$. The manifolds $Z_0$ and $Z_1$ are not properly homotopy equivalent (in fact, their ends are not properly homotopy equivalent) since they have nonisomorphic cohomology algebras at infinity~\cite{CH}. The basic idea is that both manifolds $Z_i$ have obvious splittings as ladder sums. For $Z_0$, one summand is $\R^4$, so all cup products from $H^1(Z_0;\Z/p)\otimes H^2(Z_0;\Z/p)$ are supported in the other summand in a 1-dimensional subspace of $H^3(Z_0;\Z/p)$. However, $Z_1$ has cup products on both sides, spanning a 2-dimensional subspace.
\end{example}

Our remaining examples are pairs with the same homotopy type, distinguished by more subtle means.

\begin{examples}[{\bf a}]\label{htpy} {\bf Homotopy equivalent but nonhomeomorphic sums.} It should not be surprising that the sum of two manifolds along a properly embedded line in each depends on more than just the ends and orientations involved. However, as a warm-up for end sums, we give an explicit example in \textsc{top} where moving one line changes the resulting homeomorphism type but not its proper homotopy type. Let $P$ and $Q$, respectively, denote $\C P^2$ and Freedman's fake $\C P^2$ (e.g.\ \cite{FQ}). Then there is a homotopy equivalence between $P$ and $Q$, restricting to a pairwise homotopy equivalence between the complements of a ball interior in each. But $P$ and $Q$ cannot be homeomorphic since $Q$ is unsmoothable. The ladder sum $\L(P,Q)$ is an unsmoothable topological 5-manifold with two ends. The lines $\R\times\{p\}\subset \R\times P$ and $\R\times\{q\}\subset\R\times Q$ can be chosen to lie in $\L(P,Q)$, with each spanning the two ends of $\L(P,Q)$, but they are dual to two different elements of $H^4(\L(P,Q);\Z/2)$ (cf.\ \cite{CH}), with $\R\times\{q\}$ dual to the Kirby-Siebenmann smoothing obstruction of $\L(P,Q)$. Clearly, there is a proper homotopy equivalence of $\L(P,Q)$ interchanging the two lines. Thus, the two resulting ways to sum $\L(P,Q)$ along a line with $\R\times\overline Q$ (where the orientation on $Q$ is reversed for later convenience) give properly homotopy equivalent manifolds, namely $\L(\overline Q\#P,Q)$ and $L(P,Q\#\overline Q)=L(P,P\#\overline P)$. (The last equality follows from Freedman's classification of simply connected topological 4-manifolds \cite{FQ}.) These two manifolds cannot be homeomorphic, since the latter is a smooth manifold whereas the former is unsmoothable, with Kirby-Siebenmann obstruction dual to a pair of lines running along opposite sides of the ladder. (A discussion of the cohomology of such manifolds can be found in \cite{CH}, but more simply, there are subsets $(a,b)\times Q$ on which the Kirby-Siebenmann obstruction must evaluate nontrivially.)

\noindent {\bf (b) Homotopy equivalent but nonhomeomorphic end sums.} We adapt the previous example to end sums. Instead of summing along a line, we end sum $\L(P,Q)$ with $\R\times\overline Q$ along their positive ends in two different ways (using rays obtained from the positive ends of the previous lines). We obtain a pair of properly homotopy equivalent, unsmoothable, three-ended manifolds. In one case, the modified end has a neighborhood that is smoothable, and in the other case, all three ends fail to have smoothable neighborhoods since the Kirby-Siebenmann obstruction cannot be avoided. Thus, we have a pair of nonhomeomorphic, but properly homotopy equivalent, manifolds, both obtained by an orientation-preserving end sum on the same pair of ends.

There are several other variations of the construction. We can replace the $\R$ factor by $[0,\infty)$ so that the ladder sum is one-ended,  to get an example of nonuniqueness of summing one-ended topological manifolds with compact boundary. Unfortunately, we cannot cap off the boundaries to obtain one-ended open manifolds, since the Kirby-Siebenmann obstruction is a cobordism invariant of topological 4-manifolds. However, we can modify the original ladder sum so that we do ladder surgery on the positive end, but end~sum on the negative end (which then has a neighborhood homeomorphic to $\R\times(\overline P\#\overline Q)$). Now we have a connected, two-ended open manifold whose ends can be joined by an orientation-preserving 1-handle at infinity in two different ways, yielding properly homotopy equivalent but nonhomeomorphic one-ended manifolds, only one of which has a smoothable neighborhood of infinity.

\noindent {\bf (c) Homotopy equivalent but not PL homeomorphic end sums.} In higher dimensions, the Kirby-Siebenmann obstruction of a neighborhood $V$ of an end cannot be killed by adding 1-handles at infinity (since $H^4(V;\Z/2)$ is not disturbed), but we can do the analogous construction using higher smoothing obstructions. This time, we obtain \textsc{pl} $n$-manifolds (for various $n\ge9$) that are properly homotopy equivalent but not \textsc{pl} homeomorphic. Let $P$ and $Q$ be homotopy equivalent \textsc{pl} $(n-1)$-manifolds with $P$ and $Q-\{q_0\}$ smooth but $Q$ unsmoothable. (For an explicit 24-dimensional pair, see  Anderson \cite[Proposition~5.1]{A}.) The previous discussion applies almost verbatim with \textsc{pl} in place of \textsc{top}, with the smoothing obstruction in $H^{n-1}(X;\Theta_{n-2})$ for \textsc{pl} manifolds $X$ in place of the Kirby-Siebenmann obstruction. The one change is that smoothability of $Q\#\overline Q$ follows since it is the double of the smooth manifold obtained from $Q$ by removing the interior of a  \textsc{pl} ball centered at $q_0$. (This time the orientation reversal is necessary since the smoothing obstruction need not have order 2.)
\end{examples}

\begin{examples}[{\bf a}]\label{PL} {\bf PL homeomorphic but nondiffeomorphic end sums (one-ended).}  A similar construction shows that end summing along a fixed pair of ends can produce \textsc{pl} homeomorphic but nondiffeomorphic manifolds. Let $\Sigma$ be an exotic $(n-1)$-sphere with $n>5$. Then $\Sigma$ is \textsc{pl} homeomorphic to $S^{n-1}$, so the ladder sum $\L(\Sigma,S^{n-1})$ is a two-ended smooth manifold with a \textsc{pl} self-homeomorphism that is not isotopic to a diffeomorphism. Since $\Sigma\#\overline\Sigma=S^{n-1}$, summing $\L(\Sigma,S^{n-1})$ along a line with $\R\times\overline\Sigma$ gives the two manifolds $\L(S^{n-1},S^{n-1})$ and $\L(\Sigma,\overline\Sigma)$. The first of these bounds an infinite handlebody made with 0- and 1-handles, as does its universal cover. Since a contractible 1-handlebody is a ball with some boundary points removed, it follows that the universal cover of $\L(S^{n-1},S^{n-1})$ embeds in $S^n$. However,  $\L(\Sigma,\overline\Sigma)$ contains copies of $\Sigma$ arbitrarily close to its ends. Since any homotopy $(n-1)$-sphere ($n>5$) that embeds in $S^n$ cuts out a ball, so is standard, it follows that no neighborhood of either end of  $\L(\Sigma,\overline\Sigma)$ has a cover embedding in $S^n$. Thus, the two manifolds have nondiffeomorphic ends, although they are \textsc{pl} homeomorphic. As before, we can modify this example to get a pair of end sums of two-ended manifolds, or a pair obtained from a two-ended connected manifold by joining its ends with a 1-handle in two different ways. This time however, we can also interpret the example as end summing two one-ended open manifolds, by first obtaining one-ended manifolds with compact boundary, then capping off the boundary. (Note that $\Sigma$ bounds a compact manifold. Unlike codimension-0 smoothing existence obstructions, the uniqueness obstructions are not cobordism invariants.) The resulting pair of one-ended \textsc{diff} manifolds are now easily seen to be \textsc{pl} homeomorphic (by Corollary~\ref{spherecollar}, for example) but nondiffeomorphic.

\noindent {\bf (b) Nonisotopic DIFF=PL structures on a fixed TOP 4-manifold (one-ended).} The previous construction has an analogue in dimension 4, where the categories \textsc{diff} and \textsc{pl} coincide. Replace $\R\times\Sigma$ by $W$, Freedman's exotic $\R\times S^3$. This is distinguished from the standard $\R\times S^3$ by the classical \textsc{pl} uniqueness obstruction in $H^3(\R\times S^3;\Z/2)\cong\Z/2$, dual to $\R\times\{p\}$. The ladder sum $L$ of $W$ with $\R\times S^3$ can be summed along a line with $W$ in two obvious ways. These can be interpreted as smoothings on the underlying topological manifold $\L(S^3,S^3)$, and can be transformed to an example of end summing one-ended \textsc{diff} manifolds as before: To transform $W$ into a one-ended \textsc{diff} manifold, cut it in half along a Poincar\'e homology sphere $\Sigma$, then cap it with an $E_8$-plumbing. The result $E$ is a smoothing of a punctured Freedman $E_8$-manifold. (Alternatively, we can take $E$ homeomorphic to a punctured fake $\C P^2$.) We ladder sum with $\R^4$. The two results of end summing with another copy of $E$ are identified in \textsc{top} with a ladder sum of two copies of $E$ (cf.\ Corollary~\ref{spherecollar}). The smoothings are nonisotopic (even stably, i.e., after Cartesian product with $\R^k$), since the uniqueness obstruction by which they differ near infinity is dual to a pair of lines on opposite sides of the ladder. However, the authors have not been able to distinguish their diffeomorphism types. The problem with the previous argument is that the sum of two copies of $W$ along a line is not diffeomorphic to $\R\times S^3$ (although the classical invariant vanishes). While $W$ contains a copy of $\Sigma$ separating its ends, so cannot embed in $S^4$, the sum of two copies of $W$ contains $\Sigma\#\Sigma$, which also does not embed in $S^4$. The effect of summing with reversed orientation or switched ends, or replacing $\Sigma$ by a different homology sphere, is less clear. This leads to the following question, which is discussed further in Section~\ref{Smooth} (Question~\ref{inverses2}).

\begin{ques}\label{inverses} Are there two exotic smoothings on $\R\times S^3$ whose sum along a line is the standard $\R\times S^3$?
\end{ques}

\noindent If such smoothings exist, one of which has the additional property that every neighborhood of one end has a slice $(a,b)\times S^3$ (as seen in \textsc{top}) that cannot smoothly embed in $S^4$, then the method of (a) gives two one-ended open 4-manifolds that can be end summed in two homeomorphic but not diffeomorphic (or \textsc{pl} homeomorphic) ways.
\end{examples}

\section{Uniqueness for Mittag-Leffler ends}\label{Unique}

Having examined the failure of uniqueness in the last section, we now look for hypotheses that guarantee that 1-handle attaching at infinity {\em is} unique. There are several separate issues to deal with. In the compact setting, attaching a 1-handle to given boundary components can yield two different results if both boundary components are orientable, so uniqueness requires specified orientations in that case. The same issue arises for 1-handles at infinity. Beyond that, we must consider the dependence on the involved multirays. Since rays in $\R^3$ can be knotted, uncountably many homeomorphism types of contractible manifolds arise as end sums of two copies of $\R^3$ \cite{My}. (See also \cite{CH}.) Thus, we assume more than 3 dimensions and conclude, not surprisingly, that the multirays affect the result only through their proper homotopy classes, and that the choices of (suitably oriented) tubular neighborhood maps cause no additional difficulties. We have already seen that different rays determining the same end can yield different results for end summing with another fixed manifold and ray, but we give a weak group-theoretic condition on an end that entirely eliminates dependence on the choice of rays limiting to it.

We begin with terminology for orientations. We will call an end $\epsilon$ of an $n$-manifold $X$ {\em orientable} if it has an orientable neighborhood in $X$. An orientation on one connected, orientable neighborhood of $\epsilon$ determines an orientation on every other such neighborhood, through the component of their intersection that is a neighborhood of $\epsilon$. Such a compatible choice of orientations will be called an {\em orientation of $\epsilon$}, so every orientable end has two orientations. We let $\Eo\subset\E(X)$ denote the open subset of orientable ends of $X$. (This need not be closed, as seen by deleting a sequence of points of $X$ converging to a nonorientable end.) If $\gamma$ is a singular multiray in a \textsc{diff} manifold $X$, the tangent bundle of $X$ pulls back to a trivial bundle $\gamma^*TX$ over $S\times[0,\infty)$. A fiber orientation on this bundle will be called a {\em local orientation of $X$ along $\gamma$}, and if such an orientation is specified, $\gamma$ will be called {\em locally orienting}. We apply the same terminology in \textsc{pl} and \textsc{top}, using the appropriate analogue of the tangent bundle, or equivalently but more simply, using local homology groups $H_n(X,X-\{\gamma(s,t)\})\cong\Z$. If $\gamma$ is a (nonsingular) \textsc{cat} multiray, a \textsc{cat} tubular neighborhood map $\nu$ induces a local orientation of $X$ along $\gamma$; if this agrees with a preassigned local orientation along $\gamma$, $\nu$ will be called {\em orientation preserving}. A homotopy between two singular multirays determines a correspondence between their local orientations (e.g., by pulling back the tangent bundle to the domain of the homotopy). If a singular ray $\gamma$ determines an orientable end $\epsilon_\gamma\in\Eo$, then a local orientation along $\gamma$ induces an orientation on the end, since $\gamma([k,\infty))$ lies in a connected, orientable neighborhood of $\epsilon_\gamma$ when $k$ is sufficiently large.

We now turn to the group theory of ends. See Geoghegan \cite{Ge08} for a more detailed treatment. An {\em inverse sequence of groups} is a sequence $G_1\leftarrow G_2\leftarrow G_3\leftarrow\cdots$ of groups and homomorphisms. We suppress the homomorphisms from the notation, since they will be induced by obvious inclusions in our applications. A {\em subsequence} of an inverse sequence is another inverse sequence obtained by passing to a subsequence of the groups and using the obvious composites of homomorphisms. Passing to a subsequence and its inverse procedure, along with isomorphisms commuting with the maps, generate the standard notion of equivalence of inverse sequences.

\begin{de}\label{ml} 
An inverse sequence $G_1\leftarrow G_2\leftarrow G_3\leftarrow\cdots$ of groups is called {\em Mittag-Leffler} (or {\em semistable}) if for each $i\in\Z^+$ there is a $j\ge i$ such that all $G_k$ with $k\ge j$ have the same image in $G_i$.
\end{de}

\noindent Clearly, a subsequence is Mittag-Leffler if and only if the original sequence is, so the notion is preserved by equivalences. After passing to a subsequence, we may assume $j=i+1$ in the definition.

For a manifold $X$ with a singular ray $\gamma$ and a neighborhood system $\{U_i\}$ of infinity, we reparametrize $\gamma$ so that $\gamma([i,\infty))$ lies in $U_i$ for each $i\in\Z^+$.

\begin{de} 
The {\em fundamental progroup} of $X$ based at $\gamma$ is the inverse sequence of groups $\pi_1(U_i,\gamma(i))$, where the homomorphism $\pi_1(U_{i+1},\gamma(i+1))\to \pi_1(U_i,\gamma(i))$ is the inclusion-induced map to $\pi_1(U_i,\gamma(i+1))$ followed by the isomorphism moving the base point to $\gamma(i)$ along the path $\gamma|[i,i+1]$.
\end{de}

\noindent This only depends on the \textsc{top} structure of $X$. Passing to a subsequence of $\{U_i\}$ replaces the fundamental progroup by a subsequence of it. Since any two neighborhood systems of infinity have interleaved subsequences, the fundamental progroup is independent, up to equivalence, of the choice of neighborhood system. It is routine to check that it is similarly preserved by any proper homotopy of $\gamma$, so it only depends on $X$ and the proper homotopy class of $\gamma$. Furthermore, the inverse sequence is unchanged if we replace each $U_i$ by its connected component containing $\gamma([i,\infty))$, so it is equivalent to use a neighborhood system of the end $\epsilon_\gamma$. Beware, however, that even if there is only one end, the choice of proper homotopy class of $\gamma$ can affect the fundamental progroup, and even whether its inverse limit vanishes. (See \cite[Example 16.2.4]{Ge08}. The homomorphisms in the example are injective, but changing $\gamma$ conjugates the resulting nested subgroups, changing their intersection.)

We call the pair $(X,\gamma)$ {\em Mittag-Leffler} if its fundamental progroup is Mittag-Leffler. We will see in Lemma~\ref{ML}(a) below that this condition implies $\gamma$ is determined up to proper homotopy by its induced end $\epsilon_\gamma$, so the fundamental progroup of $\epsilon_\gamma$ is independent of $\gamma$ in this case, and it makes sense to call  $\epsilon_\gamma$ a {\em Mittag-Leffler end}. Note that this condition rules out ends made by ladder surgery, and hence the examples of Section~\ref{Nonunique}. We will denote the set of Mittag-Leffler ends of $X$ by $\EML\subset\E(X)$, and its complement by $\Ebad$.

Many important types of ends are Mittag-Leffler. {\em Simply connected} ends are (essentially by definition) the special case for which the given images all vanish. {\em Topologically collared} ends, with a neighborhood homeomorphic to $\R\times M$ for some compact $(n-1)$-manifold $M$, are {\em stable}, the special case for which the fundamental progroup is equivalent to an inverse sequence with all maps isomorphisms. Other important ends are neither simply connected nor collared, but still Mittag-Leffler if the maps are nontrivial surjections (Example~\ref{surj}). Any end admits a neighborhood system for which the maps are not even surjective, obtained from an arbitrary system by adding 1-handles to each $U_i$ inside $U_{i-1}$; such ends may still be Mittag-Leffler. In the smooth category, we can analyze ends using a Morse function $\varphi$ that is exhausting (i.e., proper and bounded below). For such a function, the preimages $\varphi^{-1}(i,\infty)$ for $i\in\Z^+$ form a neighborhood system of infinity.

\begin{prop}\label{Morse} Let $X$ be a \textsc{diff} open $n$-manifold. If an end $\epsilon$ of $X$ is not Mittag-Leffler, then for every exhausting Morse function $\varphi$ on $X$ and every $t\in\R$, there are infinitely many critical points of index $n-1$ in the component of $\varphi^{-1}(t,\infty)$ containing $\epsilon$. In particular, if $X$ admits an exhausting Morse function with only finitely many index-$(n-1)$ critical points, then all of its ends are Mittag-Leffler.
\end{prop}

\begin{proof} After perturbing $\varphi$ and composing it with an orientation-preserving diffeomorphism of $\R$, we can assume each  $\varphi^{-1}[i,i+1]$ is an elementary cobordism. Since $\epsilon$ is not Mittag-Leffler, its corresponding fundamental progroup must have infinitely many homomorphisms that are not surjective. Thus, there are infinitely many values of $i$ for which $\varphi^{-1}[i,\infty)$ is made from $\varphi^{-1}[i+1,\infty)$ by attaching a 1-handle with at least one foot in the component of the latter containing $\epsilon$. This handle corresponds to an index-1 critical point of $-\varphi$, or an index-$(n-1)$ critical point of $\varphi$.
\end{proof}

The Mittag-Leffler condition on an end of a \textsc{cat} manifold is determined by its underlying \textsc{top} structure (in fact, by its proper homotopy type), so we are free to change the smooth structure on a manifold before looking for a suitable Morse function. This is especially useful in dimension 4. For example, an exhausting Morse function on an exotic $\R^4$ with nonzero Taylor invariant must have infinitely many index-3 critical points \cite{T}, but after passing to the standard structure, there is such a function with a unique critical point. (Furthermore, an exotic $\R^4$ is topologically collared and simply connected at infinity.) Proposition~\ref{Morse} is most generally stated in \textsc{top}, using topological Morse functions. (These are well-behaved \cite{KS} and can be constructed from handle decompositions, which exist on all open  \textsc{top} manifolds, e.g.\ \cite{FQ}.)

Since every Stein manifold of complex dimension $m$ (real dimension $2m$) has an exhausting Morse function with indices at most $m$, we conclude:

\begin{cor}\label{stein} For every Stein manifold of complex dimension at least 2, the unique end of each component is Mittag-Leffler. \qed
\end{cor}

\begin{example}\label{surj} For infinite-type Stein surfaces ($m=2$), the ends must be Mittag-Leffler, but they are typically neither simply connected nor stable (hence, not topologically collared). This is more generally typical for open 4-manifolds whose exhausting Morse functions require infinitely many critical points, but none of index above 2. As a simple example, let $X$ be an infinite end sum of $\R^2$-bundles over $S^2$. (Its diffeomorphism type is independent of the choice of rays, by Theorems~\ref{main} and~\ref{maincancel}, but it is convenient to think of the bundles as indexed by $\Z^+$ and summed consecutively.) If each Euler number is less than $-1$, then $X$ will be Stein. We get a neighborhood system of infinity with each $U_i$ obtained from a collar of the end of the first $i$-fold sum by attaching the remaining (simply connected) summands. Then each group $G_i$ is a free product of $i$ cyclic groups, and each homomorphism is surjective, projecting out one factor. The inverse limit is not finitely generated, so the end is not stable. (Every neighborhood system of the end has a subsequence that can be interleaved by some of our neighborhoods $U_i$.)
\end{example}

We can now state our main theorem on uniqueness of attaching 1-handles. Its primary conclusion is that when we attach 1-handles at infinity, any locally orienting defining ray that determines a Mittag-Leffler end will affect the outcome only through the end and local orientation it determines. If the end is also nonorientable, then even the local orientation has no influence (as for a compact 1-handle attached to a nonorientable boundary component). To state this in full generality, we also allow rays determining ends that are not Mittag-Leffler, which are required to remain in a fixed proper homotopy class. That is, we allow an arbitrary multiray $\gamma$, but require its restriction to the subset $\epsilon^{-1}_{\gamma}(\Ebad)$ of the index set $S$ (corresponding to rays determining ends that are not Mittag-Leffler) to lie in a fixed proper homotopy class. For each 1-handle with at least one defining ray determining a nonorientable Mittag-Leffler end, no further constraint is necessary, but otherwise we keep track of orientations. We do this through orientations of the end if they exist. In the remaining case, the end is not Mittag-Leffler, and we compare the local orientations of the rays through a proper homotopy. More precisely, we  have:

\begin{thm}\label{main} For a \textsc{cat} $n$-manifold $X$ with $n\ge4$, discrete $S$ and $i=0,1$, let $\gamma^-_i,\gamma^+_i\co S\times [0,\infty)\emb X$ be locally orienting \textsc{cat} multirays whose images (for each fixed $i$) are disjoint, and whose end functions $\epsilon_{\gamma^\pm_i}\co S\to \E(X)$ are independent of $i$. Suppose that
\begin{itemize}
\item[(a)] after $\gamma^-_0$ and $\gamma^-_1$ are restricted to the index subset $\epsilon^{-1}_{\gamma^-_0}(\Ebad)$, there is a proper homotopy between them.
\item[(b)] for each $s\in \epsilon^{-1}_{\gamma^-_0}(\Ebad\cup\Eo)\cap \epsilon^{-1}_{\gamma^+_0}(\Ebad\cup\Eo)$, the local orientations of the corresponding rays in $\gamma^-_0$ and $\gamma^-_1$ induce the same orientation of the end if there is one, and otherwise correspond under the proper homotopy of (a).
\item[(c)] the two analogous conditions apply to $\gamma^+_i$.
\end{itemize}
Let $Z_i$ be the result of attaching 1-handles to $X$ along $\gamma^\pm_i$ (for any choice of orientation-preserving tubular neighborhood maps $\nu^\pm_i$). Then there is a \textsc{cat} homeomorphism from $Z_0$ to $Z_1$ sending the submanifold $X$ onto itself by a \textsc{cat} homeomorphism \textsc{cat} ambiently isotopic in $X$ to the identity map.
\end{thm}

It follows that 1-handle attaching is not affected by reparametrization of the rays (a proper homotopy), or changing the auxiliary diffeomorphisms $\varphi^\pm$ and $\rho^\pm$ occurring in Definition~\ref{onehandles} (which only results in changing the parametrization and tubular neighborhood maps, respectively).

\begin{cor}\label{maincor} For an oriented \textsc{cat} $n$-manifold $X$ with $n\ge4$, every countable multiset of (unordered) pairs of Mittag-Leffler ends canonically determines a \textsc{cat} manifold obtained from $X$ by attaching 1-handles at infinity to those pairs of ends, respecting the orientation. \qed
\end{cor}

Since the end of $\R^n$ is Mittag-Leffler, we immediately obtain cancellation of 0/1-handle pairs at infinity:

\begin{cor}\label{cancel} For $n\ge4$, every end sum of a \textsc{cat} $n$-manifold $X$ with $\R^n$ (or countably many copies of $\R^n$) is \textsc{cat} homeomorphic to $X$. \qed
\end{cor}

\noindent See Section~\ref{slides} for further discussion of 0-handles at infinity. This corollary shows that end summing with an exotic $\R^4$ doesn't change the homeomorphism type of a smooth 4-manifold (although it typically changes its diffeomorphism type); cf.\ Section~\ref{Smooth}. It also shows:

\begin{cor}\label{spherecollar} Suppose $X_0$ and $X_1$ are connected, oriented \textsc{cat} $n$-manifolds with $n\ge4$, and that $X_0$ has an end $\epsilon$ that is \textsc{cat} collared by $S^{n-1}$. Then all manifolds obtained as the oriented end sum of $X_0$ with $X_1$ at the end $\epsilon$ are \textsc{cat} homeomorphic.
\end{cor}

\begin{proof}
Write $X_0$ as a connected sum $X\#\R^n$. Then any such end sum is $X\# X_1$.
\end{proof}

The following corollary shows that 1-handles at infinity respect Stein structures. This  will be applied to 4-manifold smoothing theory in Theorem~\ref{JB}.

\begin{cor}\label{stein1h} Every manifold $Z$ obtained from a Stein manifold $X$ by attaching 1-handles at infinity, respecting the complex orientation, admits a Stein structure. The resulting almost-complex structure on $Z$ can be assumed to restrict to the given one on $X$, up to homotopy.
\end{cor}

\begin{proof} Since every open, oriented surface has a Stein structure and a contractible space of almost complex structures, we assume $X$ has real dimension $2m\ge4$. Since $X$ is Stein, it has an exhausting Morse function with indices at most $m$. It can then be described as the interior of a smooth (self-indexed) handlebody whose handles have index at most $m$. This is well-known when there are only finitely many critical points. A proof of the infinite case is given in the appendix of \cite{yfest}, which also shows that when $m=2$ one can preserve the extra framing condition that arises for 2-handles, encoding the given almost-complex structure. By Corollaries~\ref{stein} and \ref{maincor}, we can realize the 1-handles at infinity by attaching compact handles to the handlebody before passing to the interior (and after adding infinitely many canceling 0-1 pairs if necessary to accommodate infinitely many new 1-handles, avoiding compactness issues). Now we can convert the handlebody interior back into a Stein manifold by Eliashberg's Theorem; see \cite{CE}. The almost-complex structures then correspond by construction.
\end{proof}

The proof of Theorem~\ref{main} follows from two lemmas. The first guarantees that (a) Mittag-Leffler ends are well-defined and (b) singular multirays with a given Mittag-Leffler end function are unique up to proper homotopy.

\begin{lem}[a]\label{ML} If $(X,\gamma)$ is a Mittag-Leffler pair, then every singular ray determining the same end as $\gamma$ is properly homotopic to $\gamma$. In particular, the Mittag-Leffler condition for ends is independent of choice of singular ray, so the subset $\EML\subset\E$ is well-defined.
\item[(b)] Let $\gamma_0,\gamma_1\co S\times [0,\infty)\emb X$ be locally orienting singular multirays with the same end function. Suppose that this function $\epsilon_{\gamma_0}=\epsilon_{\gamma_1}$ has image in $\EML$, and that for each $s$ with $\epsilon_{\gamma_0}(s)\in\Eo$, the corresponding locally orienting singular rays of $\gamma_0$ and $\gamma_1$ induce the same orientation (depending on $s$) of the end $\epsilon_{\gamma_0}(s)$. Then there is a proper homotopy from $\gamma_0$ to $\gamma_1$, respecting the given local orientations.
\end{lem}

The first sentence and its converse are essentially Proposition~16.1.2 of \cite{Ge08}, which is presented as an immediate consequence of two earlier statements: Proposition~16.1.1 asserts that the set of proper homotopy classes of singular rays approaching an arbitrary end corresponds bijectively to the derived limit $\lim^1_\leftarrow\pi_1(U_i,\gamma(i))$ of a neighborhood system $U_i$ of infinity; Theorem~11.3.2 asserts that an inverse sequence of countable groups $G_i$ is Mittag-Leffler if and only if $\lim^1_\leftarrow G_i$ has only one element. We follow those proofs but considerably simplify the argument, eliminating use of derived limits, by focusing on the Mittag-Leffler case. This reveals the underlying geometric intuition: If an end $\epsilon$ is topologically collared by a neighborhood identified with $\R\times M$, and $\gamma=(\gamma_\R,\gamma_M)\co[0,\infty)\to\R\times M$ is a singular ray, we can assume after a standard proper homotopy of the first component that $\gamma_\R\co[0,\infty)\to\R$ is inclusion. Then the proper homotopy $\gamma_s(t)=(t,\gamma_M((1-s)t))=\frac{1}{1-s}\gamma((1-s)t)$ (where the last multiplication acts only on the first factor) stretches the image of $\gamma$, pushing any winding in $M$ out toward infinity, so that when $s\to1$ the ray becomes a standard radial ray. If, instead, $\epsilon$ only has a neighborhood system with $\pi_1$-surjective inclusions, we can compare two singular rays using an initial proper homotopy after which they agree on $\Z^+\subset[0,\infty)$, and so only differ by a proper sequence of loops. Then $\pi_1$-surjectivity again allows us to push the differences out to infinity: inductively collapse loops by transferring their homotopy classes to more distant neighborhoods of infinity, so that the resulting homotopy sends one ray to the other. In the general Mittag-Leffler case, we still have enough surjectivity to push each loop to infinity after pulling it back a single level in the neighborhood system (with properness preserved because we only pull back one level). The following proof efficiently encodes this procedure with algebra.

\begin{proof} First we prove (a), showing that an arbitrary singular ray $\gamma'$ determining the same Mittag-Leffler end as $\gamma$ is properly homotopic to it. We also keep track of preassigned local orientations along the two singular rays. If $\epsilon_\gamma$ is orientable, we assume these local orientations induce the same orientation on $\epsilon_\gamma$ (as in (b)). Let $\{U_i\}$ be a neighborhood system of infinity, arranged (by passing to a subsequence if necessary) so that each $j$ is $i+1$ in the definition of the Mittag-Leffler condition, and that the component of $U_1$ containing $\epsilon_\gamma$ is orientable if $\epsilon_\gamma$ is. Then reparametrize $\gamma$ so that each $\gamma([i,\infty))$ lies in $U_i$. Reparametrize $\gamma'$ similarly, then arrange it to agree with $\gamma$ on $\Z^+$ by inductively moving $\gamma'$ near each $i\in\Z^+$ separately, with compact support inside $U_i$. The limiting homotopy is then well-defined and proper. If $\epsilon_\gamma$ is nonorientable, then so is the relevant component of each $U_i$, so we can assume (changing the homotopy via orientation-reversing loops as necessary) that the local orientations along the two singular rays agree at each $i$. (This is automatic when $\epsilon_\gamma$ is orientable.) The two singular rays now differ by a sequence of orientation-preserving loops, representing classes $x_i\in\pi_1(U_i,\gamma(i))$ for each $i\ge 1$. Inductively choose orientation-preserving classes $y_i\in\pi_1(U_i,\gamma(i))$ for all $i\ge2$ starting from an arbitrary $y_2$, and for $i\ge1$ choosing $y_{i+2}\in\pi_1(U_{i+2},\gamma(i+2))$ to have the same image in $\pi_1(U_i,\gamma(i))$ as $x_{i+1}^{-1}y_{i+1}\in\pi_1(U_{i+1},\gamma(i+1))$. (This is where the Mittag-Leffler condition is necessary.) For each $i\ge 1$, let $z_i=x_iy_{i+1}\in\pi_1(U_i,\gamma(i))$ (where we suppress the inclusion map). In that same group, we then have $z_iz_{i+1}^{-1}=x_iy_{i+1}y_{i+2}^{-1}x_{i+1}^{-1}=x_i$. After another proper homotopy, we can assume the two singular rays and their induced local orientations on $X$ agree along $\frac12\Z^+$ and give the sequence $z_1,z_2^{-1},z_2,z_3^{-1},\dots$ in $U_1,U_1,U_2,U_2,\dots$. Now a proper homotopy fixing $\Z^++\frac12$ cancels all loops between these points and eliminates $z_1$ (moving $\gamma'(0)$) so that the two singular rays coincide. This completes the proof of (a), and also (since $\EML$ is now well-defined) the case of (b) with $S$ a single point.

For the general case of (b), we wish to apply the previous case to each pair of of singular rays separately. The only issue is properness of the resulting homotopy of singular multirays. Let $\{W_j\}$ be a neighborhood system of infinity with $W_1=X$. For each $s\in S$, find the largest $j$ such that $W_j$ contains both rays indexed by $s$, and apply the previous case inside that $W_j$. Since the singular multirays are proper, each $W_j$ contains all but finitely many pairs of singular rays, guaranteeing that the combined homotopy is proper.
\end{proof}

\begin{remark} To see the correspondence of this proof with the geometric description, first consider the case with all inclusion maps $\pi_1$-surjective. Then the argument simplifies: We can just define $z_1=1$, and inductively choose $z_{i+1}$ to be any pullback of $x_i^{-1}z_i$. Then $z_i$ is a pullback of $(x_1\cdots x_{i-1})^{-1}$ to $U_i$, exhibiting the loops being transferred toward infinity.
\end{remark}

To upgrade a proper homotopy of multirays to an ambient isotopy, we need the following lemma.

\begin{lem}\label{ambient} Suppose that $X$ is a \textsc{cat} $n$-manifold with $n\ge 4$ and $Y$ is a \textsc{cat} 1-manifold with $b_1(Y)=0$. Let $\Gamma\co I\times Y\emb \inter X$ be a topological proper homotopy, between \textsc{cat} embeddings $\gamma_i$ ($i=0,1$) that extend to \textsc{cat} tubular neighborhood maps $\nu_i\co Y\times\R^{n-1}\emb X$ whose local orientations correspond under $\Gamma$. Then there is a \textsc{cat} ambient isotopy $\Phi\co I\times X\to X$, supported in a preassigned neighborhood of $\im\Gamma$, such that $\Phi_0=\id_X$ and $\Phi_1\circ\nu_0$ agrees with $\nu_1$ on a neighborhood of $Y\times\{0\}$ in $Y\times\R^{n-1}$.
\end{lem}

\noindent This lemma is well-known when \textsc{cat}=\textsc{diff} or \textsc{pl}, but a careful proof seems justified by the subtlety of noncompactness: The corresponding statement in $\R^3$ is false even with $\Gamma$ a proper (nonambient) isotopy of $Y=\R$. (Such an isotopy $\Gamma$ can slide a knot out to infinity, changing the fundamental group of the complement, and this can even be done while fixing the integer points of $\R$.) The case \textsc{cat}=\textsc{top} is also known to specialists. We did not find a theorem in the literature from which it follows immediately. Instead, we derive it from much stronger results of Dancis \cite{D76} with antecedents dating back to pioneering work of Homma \cite{H62}.

\begin{proof}  First we solve the case \textsc{cat}=\textsc{diff}. By transversality, we may assume (after an ambient isotopy that we absorb into $\Phi$) that $\gamma_0$ and $\gamma_1$ have disjoint images. Then we properly homotope $\Gamma$ rel $\partial I\times Y$ to be smooth and generic, so it is an embedding if $n\ge 5$ and an immersion with isolated double points if $n=4$. After decomposing $Y$ as a cell complex with 0-skeleton $Y_0$, we can assume $\Gamma$ restricts to a smooth embedding on some neighborhood of $I\times Y_0$. Then there is a tubular neighborhood $J$ of $Y_0$ in $Y$ such that $\Gamma|(I\times J)$ extends to an ambient isotopy. (Apply the Isotopy Extension Theorem separately in disjoint compact neighborhoods of the components of $\Gamma(I\times Y_0)$.) After using this ambient isotopy to define $\Phi$ for parameter $t\le\frac12$, it suffices to assume $\Gamma$ fixes $J$, and view $\Gamma$ as a countable collection of path homotopies of the 1-cells of $Y$.  We need the resulting immersed 2-disks to be disjoint. This is automatic when $n\ge 5$, but is the step that fails for knotted lines in $\R^3$. For $n=4$, we push the disks off of each other by finger moves. This operation preserves properness of $\Gamma$ since each compact subset of $X$ initially intersects only finitely many disks, which have only finitely many intersections with other disks (and we do not allow finger moves over other fingers). Now we can extend to an ambient isotopy, working in disjoint compact neighborhoods of the disks. We arrange $\nu_0$ to correspond with $\nu_1$ by uniqueness of tubular neighborhoods and contractibility of the components of $Y$.

We reduce the \textsc{pl} and \textsc{top} cases to \textsc{diff}. As before, we can assume the images of $\gamma_0$ and $\gamma_1$ are disjoint. (We did not find a clean \textsc{top} statement of this. However, we can easily arrange $\gamma_0(Y_0)$ to be disjoint from $\gamma_1(Y)$, then apply \cite[General Position Lemma~3]{D76}. While this lemma assumes the moved manifold is compact and without boundary, we can apply it to the remaining 1-cells of $\gamma_0(Y)$ by arbitrarily extending them to circles.) A tubular neighborhood $N$ of $\gamma_0(Y)\sqcup\gamma_1(Y)$ now inherits a smoothing $\Sigma$ from the maps $\nu_i$. If $n=4$, $\Sigma$ extends over the entire manifold $X$ except for one point in each compact component \cite{FQ}. Homotoping $\Gamma$ off of these points, we reduce to the case \textsc{cat}=\textsc{diff}. If $n\ge 5$, we again homotope $\Gamma$ rel $\partial I\times Y$ to an embedding. (Again we found no clean \textsc{top} statement, but it follows by smoothing $\Gamma$ on $\Gamma^{-1}(N)$, homotoping so that $\Gamma^{-1}(N)$ is a collar of $\partial I\times Y$, and applying \cite[Corollary~6.1]{D76} in $X-N$.) Since $(I,\partial I)\times Y$ has no cohomology above dimension 1, there is no obstruction to extending $\Sigma$ over a neighborhood of the image of $\Gamma$, again reducing to  \textsc{cat}=\textsc{diff}.
\end{proof}

\begin{proof}[Proof of Theorem~\ref{main}] For each $i=0,1$, the two multirays $\gamma^-_i$ and $\gamma^+_i$ can be thought of as a single multiray $\gamma_i$ with index set $S^*=S\times\{-1,1\}$. For each index $(s,\sigma)\in\epsilon_{\gamma_0}^{-1}(\Eo)\subset S^*$, we arrange for the corresponding locally orienting rays in $\gamma_0$ and $\gamma_1$ to induce the same orientation of the end: If this is not already true, then Hypothesis (b) of the theorem implies that the opposite end $\epsilon_{\gamma_0}(s,-\sigma)$ is Mittag-Leffler but nonorientable. In this case, reverse the local orientations along both rays in $\gamma_1$ parametrized by $s$. This corrects the orientations without changing $Z_1$, since the change extends as a reflection of the 1-handle $\{s\}\times[0,1]\times\R^{n-1}$. Now split $\gamma_i$ into two multirays $\gamma_i^{\rm ML}$ and $\gamma_i^{\rm bad}$, according to whether the rays determine Mittag-Leffler ends. By Hypothesis (a), we have a proper homotopy from $\gamma_0^{\rm bad}$ to $\gamma_1^{\rm bad}$, which respects the local orientations by Hypothesis (b) after further possible flips as above when the opposite end is Mittag-Leffler but nonorientable. Lemma~\ref{ML}(b) then gives a proper homotopy from $\gamma_0^{\rm ML}$ to $\gamma_1^{\rm ML}$ respecting local orientations. Reassembling the multirays, we obtain a proper homotopy from $\gamma_0$ to $\gamma_1$ that respects local orientations. Now we apply Lemma~\ref{ambient} with $Y=S^*\times[0,\infty)$, and $\nu_i$ the given tubular neighborhood map for $\gamma_i$ (after the above flips). We obtain a \textsc{cat} ambient isotopy $\Phi$ of $\id_X$ such that $\Phi_1\circ\nu_0$ agrees with $\nu_1$ on a neighborhood $N$ of $S^*\times[0,\infty)\times\{0\}$ in $S^*\times[0,\infty)\times\R^{n-1}$. Note that the quotient space $Z_i$ does not change if we cut back the 1-handles $S\times[0,1]\times\R^{n-1}$ to any neighborhood $N'$ of $S\times\{\frac12\}\times\R^{n-1}$ and use the restricted gluing map. Recall that the gluing map factors through an $\R^{n-1}$-bundle map $\id_S\times\varphi^\pm\times\rho^\pm$ to $S^*\times[0,\infty)\times\R^{n-1}$. We can assume that the resulting image of $N'$ lies in some disk bundle (with radii increasing along the rays) inside $S^*\times[0,\infty)\times\R^{n-1}$. A smooth ambient isotopy supported inside a larger disk bundle moves this image into $N$. Conjugating with $\nu_i$ gives a \textsc{cat} ambient isotopy $\Psi_{(i)}$ on $X$. Then $\Phi'=\Psi_{(1)}^{-1}\circ\Phi\circ\Psi_{(0)}$ is a \textsc{cat} ambient isotopy for which $\Phi'_1\circ\nu_0$ agrees with $\nu_1$ on $N'$. The \textsc{cat} homeomorphism $\Phi'_1$ extends to one sending $Z_0$ to $Z_1$ with the required properties.
\end{proof}

We can now address uniqueness of ladder surgeries. Note that their definition immediately extends to unoriented manifolds, provided that we use locally orienting multirays.

\begin{cor}\label{ladder} For a \textsc{cat} manifold $X$, discrete $S$ and $i=0,1$, let $\gamma^\pm_i\co S\times [0,\infty)\emb X$ be locally orienting \textsc{cat} multirays with disjoint images (for each fixed $i$) such that the end functions $\epsilon_{\gamma^\pm_i}\co S\to \E(X)$ are independent of $i$. Suppose that for each $s\in \epsilon^{-1}_{\gamma^-_0}(\Eo)\cap \epsilon^{-1}_{\gamma^+_0}(\Eo)$, the local orientations of the corresponding rays in $\gamma^\pm_i$  induce the same orientation of the end for $i=0,1$. Then the manifolds $Z_i$ obtained by ladder surgery on $X$ along $\gamma^\pm_i$ are \textsc{cat} homeomorphic.
\end{cor}

\begin{proof} As in the previous proof, we assume that each ray of $\gamma^\pm_0$ determining an orientable end induces the same orientation of that end as the corresponding ray of $\gamma^\pm_1$, after reversing orientations on some mated pairs of rays (with the mate determining a nonorientable end). Since the end functions are independent of $i$, there is a proper homotopy of $\gamma^\pm_0$ for each choice of sign, after which  $\gamma^\pm_i(s,n)$ is independent of $i$ for each $s\in S$ and $n\in\Z^+$ (as in the proof of Lemma~\ref{ML}). We can assume the local orientations agree at each of these points, after possibly changing the homotopy on each ray determining a nonorientable end. The proper homotopy of $\gamma^\pm_0|S\times\Z^+$ extends to an ambient isotopy as in the proof of Lemma~\ref{ambient}, without dimensional restriction (since we only deal with the 0-skeleton $Y_0$).
\end{proof}

\section{Smoothings of open 4-manifolds}\label{Smooth}

Recall from Section~\ref{handles} that end summing with an exotic $\R^4$ can be defined as an operation on the smooth structures of a fixed topological 4-manifold, and that one can similarly change smoothings of $n$-manifolds by summing with an exotic $\R\times S^{n-1}$ along a properly embedded line. (The latter is most interesting when $n=4$, but the comparison with higher dimensions is illuminating.) We now address uniqueness of both operations, expressing them as monoid actions on the set of isotopy classes of smoothings of a topological manifold. We define an {\em action} of a monoid $\M$ on a set $\Sm$ by analogy with group actions: Each element of $\M$  is assigned  a function $\Sm\to\Sm$, with the identity of $\M$  assigned $\id_\Sm$, and with monoid addition corresponding to composition of functions in the usual way.

We first consider end summing with an exotic $\R^4$. In \cite{infR4}, it was shown that the set $\r$ of oriented diffeomorphism types of smooth manifolds homeomorphic to $\R^4$ admits the structure of a commutative monoid under end sum, with identity given by the standard $\R^4$, and such that countable sums are well-defined and independent of order and grouping. (Infinite sums were defined as simultaneously end summing onto the standard $\R^4$ along a multiray in the latter. Thus, the statement follows from Theorem~\ref{main} with the two multirays $\gamma^+_i$ in $\R^4$ differing by a permutation of $S$, and with Corollary~\ref{cancel} addressing grouping; cf.\ also Section~\ref{slides}.) For any set $S$, the Cartesian product $\r^S$ inherits a monoid structure with the same properties, as does the submonoid $\r^S_c$ of $S$-tuples that are the identity except in countably many coordinates. Note that every action by such a monoid inherits a notion of infinite iteration, since we can sum infinitely many monoid elements together before applying them. In the case at hand, we obtain the following corollary of the lemmas of the previous section. We again split a multiray $\gamma\co S\times[0,\infty)\to X$ into two multirays $\gML\co \SML\times[0,\infty)\to X$ and $\gbad\co \Sbad\times[0,\infty)\to X$, according to which rays determine Mittag-Leffler ends.

\begin{cor}\label{R4} Let $X$ be a \textsc{top} 4-manifold with a locally orienting \textsc{top} multiray $\gamma\co S\times[0,\infty)\to X$. Then $\gamma$ determines an action of $\r^S$ on the set $\Sm(X)$ of isotopy classes of smoothings of $X$. The action only depends on the proper homotopy class of the locally orienting multiray $\gbad$, the function $\epsilon_{\gML}$, and the subset of $\SML$ inducing a preassigned orientation on the orientable ends. In particular, if $X$ is oriented (or orientations are specified on all orientable Mittag-Leffler ends) then the monoid $\r^{\EML(X)}_c$ acts canonically on $\Sm(X)$.
\end{cor}

\noindent Note that orientation reversal induces an involution on  the monoid $\r$, and changing the local orientations of $\gamma$ changes the action by composing with this involution on the affected factors of $\r^S$.

\begin{proof} To define the action, fix a smoothing on $X$ and an indexed set $\{R_s|\thinspace s\in S\}$ of elements of $\r$. According to Quinn (\cite{Q}, cf.\ also \cite{FQ}), $\gamma$ can be made smooth by a  \textsc{top} ambient isotopy. For each $s\in S$, choose a smooth ray $\gamma'$ in $R_s$, and use it to sum $R_s$ with $X$ along the corresponding ray in $X$. We do this by homeomorphically identifying the complement of a tubular neighborhood of $\gamma'$ (with smooth $\R^3$ boundary) with a corresponding closed tubular neighborhood of the ray in $X$ (preserving orientations), then transporting the smoothing of $R_s$ to $X$. We assume the identification is smooth near each boundary $\R^3$, and then the smoothing fits together with the given one on the rest of $X$. This process can be performed simultaneously for all $s\in S$, provided that we work within a closed tubular neighborhood of $\gamma$. Each ray $\gamma'$ is unique up to smooth ambient isotopy (Lemma~\ref{ambient}), and the required identifications of neighborhoods (homeomorphic to the half-space $[0,\infty)\times\R^3$) are unique up to topological ambient isotopy that is smooth on the boundary (by the Alexander trick), so the resulting isotopy class of smoothings on $X$ is independent of choices made in the $R_s$ summands. Similarly, the resulting smoothing is changed by an isotopy if the original smoothing of $X$ is isotoped or $\gamma$ is changed by a proper homotopy (Lemma~\ref{ambient} again). In particular, the initial choice of smoothing of $\gamma$ does not matter. Since the proper homotopy class of the locally orienting multiray $\gML$ is determined by $\epsilon_{\gML}$ and the orientation data (Lemma~\ref{ML}(b)), we have a well-defined function $\Sm(X)\to\Sm(X)$ determined by an element of $\r^S$ and the data given in the corollary.

The rest of the corollary is easily checked. To verify that we have a monoid action, consecutively apply two elements $\{R_s\}$ and $\{R'_s\}$ of $\r^S$. This uses the multiray $\Gamma$ twice. After summing with  each $R_s$, however, $\Gamma$ lies in the new summands, so we are equivalently end summing $X$ with the sum of the two elements of $\r^S$ as required. If we enlarge the index set $S$ of $\{R_s\}$ while requiring all of the new summands $R_s$ to be $\R^4$, the induced element of $\Sm(X)$ will be unchanged, so it is easy to deduce the last sentence of the corollary even when $\EML$ is uncountable.
\end{proof}

In contrast with more general end sums, the action of $\r^S$ on $\Sm(X)$ is not known to vary with the choice of proper homotopy class of $\gamma$ (for a fixed end function).

\begin{ques} Suppose that two locally orienting multirays in $X$ have the same end function, and that for each $s\in S$, the two corresponding rays induce the same orientation on the corresponding end, if it admits one. Can the two actions of $\r^S$ on $\Sm(X)$ be different?
\end{ques}

\noindent We can also ask about diffeomorphism types rather than isotopy classes; cf.\ Question~\ref{R4sum}. Clearly, any example of nonuniqueness must involve an end that fails to be Mittag-Leffler, such as one arising by ladder surgery. While such examples seem likely to exist, there are also reasons for caution, as we now discuss.

First, not every exotic $\R^4$ can give such examples. Freedman and Taylor \cite{FT} constructed a ``universal" $\R^4$, $R_U\in\r$, which is characterized as being the unique fixed point of the $\r$-action on itself. They essentially showed that for any smoothing $\Sigma$ of a 4-manifold $X$, the result of end summing with copies of $R_U$ depends only on the subset of $\E(X)$ at which the sums are performed, regardless of whether those ends are Mittag-Leffler. Then $\r$ subsequently acts trivially on each of those ends. They also showed that the result of summing with $R_U$ on a dense subset of ends creates a smoothing depending only on the stable isotopy class of $\Sigma$ (classified by $H^3(X,\partial X;\Z/2)$). For such a smoothing, $\r^S$ acts trivially for any choice of multiray. The main point is that the universal property is obtained through a countable collection of disjoint compact subsets of $R_U$ that allow h-cobordisms to be smoothly trivialized. If $X$ is summed with $R_U$ on one side of a ladder sum (for example), those compact subsets are also accessible on the other side by reaching through the rungs of the ladder.

A second issue is that examples of nonuniqueness would be subtle and hard to distinguish:

\begin{prop}\label{subtle_prop} Let $X$ be a \textsc{top} 4-manifold with smoothing $\Sigma$. Let $\gamma_0,\gamma_1\co S\times[0,\infty)\to X$ be multirays as in the above question, inducing smoothings $\Sigma_0$ and $\Sigma_1$, respectively, via a fixed element of $\r^S$. Then for every compact \textsc{diff} 4-manifold $K$, every $\Sigma_0$-smooth embedding $\iota\co K\to X$ is \textsc{top} ambiently isotopic to a $\Sigma_1$-smooth embedding. After isotopy of $\Sigma_1$, every neighborhood of infinity in $X$ contains another such neighborhood $U$ such that whenever $\iota(K)\subset U$ and $K$ is a 2-handlebody, the resulting isotopy can be assumed to keep $\iota(K)$ inside $U$.
\end{prop}

\noindent This shows that many of the standard 4-dimensional techniques for distinguishing smooth structures will fail in the above situation. One of the oldest techniques for distinguishing two smoothings on $\R^4$ is to find a compact \textsc{diff} manifold that smoothly embeds in one but not the other \cite{infR4}. A newer incarnation of this idea is the Taylor invariant \cite{T}, distinguishing \textsc{diff} 4-manifolds via an exotic $\R^4$ embedded in one with compact closure. Clearly, such techniques must fail in the current situation. Most recently, the genus function has turned out to be useful \cite{MinGen}, distinguishing by the minimal genera of smoothly embedded surfaces representing various homology classes. However, any such surface for $\Sigma_0$ will be homologous to one of the same genus for $\Sigma_1$ and vice versa. Minimal genera at infinity \cite{MinGen} will also fail: If we choose a system of neighborhoods $U$ of infinity as in the proposition, any corresponding sequence of $\Sigma_0$-smooth surfaces in these will be homologous to a corresponding sequence for $\Sigma_1$ with the same genera. A possibility remains of distinguishing $\Sigma_0$ and $\Sigma_1$ by sequences of smoothly embedded 3-manifolds approaching infinity (such as by the engulfing index of \cite{BG}, cf.\ also Remark~4.3(b) of \cite{MinGen}), but there does not currently seem to be any good way to analyze such sequences. Note that the situation is not improved by passing to a cover, since the corresponding lifted smoothings will behave similarly. (The multirays $\gamma_i$ will lift to multirays, and for each $s\in S$ the lifts of the corresponding rays of $\gamma_0$ and $\gamma_1$ will will be multirays with end functions whose images have the same closure in $\E(\widetilde X)$, cf.\ last paragraph of proof of Theorem~8.1 in \cite{MinGenV2}. The proof below still applies to this situation.)

\begin{proof} For the first conclusion, let $\overline{\nu}_i\co S\times[0,\infty)\times D^3\to X$ be the closed tubular neighborhood maps of the multirays $\gamma_i$ used for the end sums. By properness, both subsets $\overline{\nu}_i^{-1}\iota(K)$ are contained in a single subset of the form $T=S_0\times[0,N]\times D^3$ for some finite $S_0\subset S$ and $N\in\Z^+$. We need a $\Sigma$-smooth ambient isotopy $\Phi_t$ of $\id_X$ such that $\Phi_1\circ\overline{\nu}_0=\overline{\nu}_1$ on $T$, allowing no new intersections with $\iota(K)$, i.e., with $\overline{\nu}_1^{-1}\Phi_1\iota(K)$ still lying in $T$. This is easily arranged, since for each $s\in S_0$ the corresponding rays of $\gamma_0$ and $\gamma_1$ determine the same end and induce the same orientation on it if possible. This allows us to move $\gamma_0(s,N)$ to $\gamma_1(s,N)$ so that the local orientations agree, and then complete the isotopy following the initial segments of the rays. (The end hypothesis is needed when $X-\iota(K)$ is disconnected, for example.) After we perform the end sums, our isotopy will only be topological. However, $\Phi_1\circ\iota$ will be $\Sigma_1$-smooth as required, since the new smoothings correspond under $\Phi_1$ on the images of $T$ and the smoothing $\Sigma$ is preserved elsewhere on $\iota(K)$.

For the second statement, assume (isotoping $\Sigma_1$) that the images of $\overline{\nu}_i$, for $i=0,1$, are disjoint. Given a neighborhood of infinity, pass to a smaller neighborhood $U$ such that the two subsets $\overline{\nu}_i^{-1}(U)$ are equal, with complement of the form $S_1\times[0,N']\times D^3$ for some finite $S_1$ and $N'\in\Z^+$. For any $K$ and $\iota$ with $\iota(K)\subset U$, we can repeat the previous argument. There is only one difficulty: If $K=M^3\times I$, for example, some sheets of $M$ may be caught between $\partial U$ and the moving image of $\gamma_0$ during the final isotopy, and be pushed out of $U$. However, if $K$ is a handlebody with all indices 2 or less, we can remove the image of $K$ from the path of $\gamma_0$ (which will be following arcs of $\gamma_1$) by transversality. The statement now follows as before.
\end{proof}

Elements of $\r$ can be either {\em large} or {\em small}, depending on whether they contain a compact submanifold that cannot smoothly embed in the standard $\R^4$ (e.g., \cite[Section~9.4]{GS}). Action on $\Sm(X)$ by small elements does not change the invariants discussed above (except for 3-manifolds at infinity),  but still can yield uncountably many diffeomorphism types \cite[Theorem~7.1]{MinGen}. However, large elements typically do change invariants. In particular, the minimal genus of a homology class can drop under end sum with, for example, the universal $\R^4$ \cite[Theorem~8.1]{MinGen}. For Stein surfaces, the adjunction inequality gives a lower bound on minimal genera, which is frequently violated after such sums. Thus, the following application of Corollary~\ref{stein1h} seems surprising:

\begin{thm}[Bennett]\label{JB} \cite[Corollary~4.1.3]{BenDis} There is a family $\{R_t|\thinspace t\in\R\}$ of distinct large elements of $\r$ (with nonzero Taylor invariant) such that if $Z$ is obtained from a Stein surface $X$ by any orientation-preserving end sums with elements $R_t$ then the adjunction inequality of $X$ applies in $Z$.
\end{thm}

\noindent Nevertheless, we expect such sums to destroy the Stein structure, since every handle decomposition of each $R_t$ requires infinitely many 3-handles. The idea of the proof is that \cite{BenDis} or \cite{Ben} constructs such manifolds $R_t$ embedded in Stein surfaces, in such a way that the sums can be performed pairwise. By Corollary~\ref{stein1h}, we obtain $Z$ embedded in a Stein surface so that the adjunction inequality is preserved.

Next we consider sums along properly embedded lines. For a fixed $n\ge 4$, let $\Q$ denote the set of oriented diffeomorphism types of manifolds homeomorphic to $\R\times S^{n-1}$, with a given ordering of their two ends. Each such manifold admits a \textsc{diff} proper embedding of a line, preserving the order of the ends, and this is unique up to \textsc{diff} ambient isotopy by Lemma~\ref{ambient}. Thus, $\Q$ has a well-defined commutative monoid structure induced by summing along lines, preserving orientations on the lines and $n$-manifolds. (This time, properness prevents infinite sums.) The identity is $\R\times S^{n-1}$ with its standard smoothing. For $n=5,6,7$, $\Q$ is trivial, and for $n>5$, $\Q$ is canonically isomorphic to the finite group $\Theta_{n-1}$ of homotopy $(n-1)$-spheres \cite{KM} (by taking their product with $\R$). However when $n=4$, $\Q$ has much more structure: High-dimensional theory predicts that $\Q$ should be $\Z/2$, but in fact it is an uncountable monoid with an epimorphism to $\Z/2$ (analogous to the Rohlin invariant of homology 3-spheres). Uncountability is already suggested by Corollary~\ref{R4}, but the structure of $\Q$ is richer than can be obtained just by acting by $\r$ at the two ends, as can be seen as follows. For $V,V'\in\Q$, call $V$ a {\em slice} of $V'$ if it embeds in $V'$ separating the ends. (For this discussion, orientations and order of the ends do not matter.) Every known ``large" exotic $\R^4$ has a neighborhood of infinity in $\Q$ with the property that disjoint slices are never diffeomorphic \cite{infR4}. This neighborhood clearly has infinitely many disjoint slices, which comprise an infinite family in $\Q$ such that no two share a common slice. Thus, no two are obtained from a common element of $\Q$ by the action of $\r\times\r$. A similar family representing the other class in $\Z/2$ is obtained from the end of a smoothing of Freedman's punctured $E_8$-manifold.

To get an action on $\Sm(X)$ for $n\ge4$, let $\gamma\co S\times\R\to X$ ($S$ discrete) be a proper, locally orienting \textsc{top} embedding. Then $\Q^S$ has a well-defined action on $\Sm(X)$ (although without infinite iteration) by the same method as before, and this only depends on the proper homotopy class of $\gamma$. (We assume after proper homotopy that $\gamma^{-1}(\partial X)=\emptyset$. To see that a self-homeomorphism rel boundary of $\R\times D^{n-1}$ is isotopic to the identity, first use the topological Schoenflies Theorem to reduce to the case where $\{0\}\times D^{n-1}$ is fixed.) Note that while $\Q$ admits only finite sums, the set $S$ may be countably infinite. Examples~\ref{PL} showed that the action of $\Q$ on $\Sm(X)$ for a two-ended 4-manifold $X$ can depend on the choice of line spanning the ends, and in high dimensions, even the resulting diffeomorphism type can depend on the line. We next find fundamental group conditions eliminating such dependence.

To obtain such conditions, note that the fundamental progroup of $X$ based at a ray $\gamma$ has an inverse limit with well-defined image in $\pi_1(X,\gamma(0))$. In the Mittag-Leffler case, its image equals the image of $\pi_1(U_2,\gamma(2))$ for a suitably defined neighborhood system of infinity (i.e.\ with $j=i+1$ in Definition~\ref{ml}). If $\gamma$ is instead a line, it splits as a pair $\gamma_\pm$ of rays, obtained by restricting its parameter $\pm t$ to $[0,\infty)$, determining ends $\epsilon_\pm$ and images $G_\pm\subset\pi_1(X,\gamma(0))$ of the corresponding inverse limits. We will call the pair $(\epsilon_-,\epsilon_+)$ a {\em Mittag-Leffler couple} if both ends are Mittag-Leffler and the double coset space $G_-\backslash\pi_1(X,\gamma(0))/G_+$ is trivial. The proof below shows that $\gamma$ is then uniquely determined up to proper homotopy by the pair of ends, so the condition is independent of choice of $\gamma$ (as well as the direction of $\gamma$). A proper embedding $\gamma\co S\times\R\to X$ now splits into $\gML$ and $\gbad$ according to which lines connect Mittag-Leffler couples, and the restriction $\epsilon_{\gML}$ of the end function $\epsilon_{\gamma}\co S\times\{\pm1\}\to \E$ picks out the corresponding pairs of Mittag-Leffler ends. For simplicity, we now assume $X$ is oriented.

\begin{cor}\label{SxR} Let $X$ be an oriented topological $n$-manifold ($n\ge4$) with a  proper embedding $\gamma\co S\times\R\to X$. Then $\gamma$ determines an action of $\Q^S$ on $\Sm(X)$, depending only on the proper homotopy classes of $\gbad$ and $\gML$. If the latter consists of finitely many lines, it only affects the action through its end function $\epsilon_{\gML}$.
\end{cor}

\noindent If $X$ is simply connected and $\EML$ is finite, we obtain a canonical action of $\Q^{\EML\times\EML}$ on $\Sm(X)$.

\begin{proof} For a proper embedding $\gamma$ of $\R$ determining a Mittag-Leffler couple $\epsilon_\pm$ as above, we show that any other embedding $\gamma'$ determining the same ordered pair of ends is properly homotopic to $\gamma$. This verifies that Mittag-Leffler couples are well-defined, and proves the corollary. (The finiteness hypothesis guarantees properness of the homotopy that we make using the proper homotopies of the individual lines.) Let $\{U_i\}$ be a neighborhood system of infinity as in the proof of Lemma~\ref{ML}, and reparametrize the four rays $\gamma_\pm$ and$\gamma'_\pm$  accordingly (fixing 0). As before, we can properly homotope $\gamma'$ to agree with $\gamma$ on $\Z\subset\R$, so that $\gamma$ and $\gamma'$ are related by a doubly infinite sequence of loops. The loop captured between $\pm2$ (starting at $\gamma(0)$, then following $\gamma_-$, $\gamma'$ and (backwards) $\gamma_+$) represents a class in $\pi_1(X,\gamma(0))$ that by hypothesis can be written in the form $w_-w_+$ with $w_\pm\in G_\pm$. After a homotopy of $\gamma'$ supported in $[-2,2]$, we can assume that $\gamma'=\gamma$ on $[-1,1]$, and the innermost loops are given by $w_\pm$ pulled back 
 to $\pi_1(U_1,\gamma(\pm1))$. Working with each sign separately, we now complete the proof of Lemma~\ref{ML}(a), denoting the pullback of $w_\pm$ by $x_1$ as before. By the definition of $G_\pm$, $x_1$ can be assumed to pull back further to $\pi_1(U_2,\gamma_\pm(2))$; let $y_2$ be the inverse of such a pullback. Completing the construction, we see that $z_1=1$, so that $\gamma'$ is then properly homotoped to $\gamma$ rel $[-1,1]$.
\end{proof}

Corollary~\ref{SxR} is most interesting when $n=4$, since classical smoothing theory reduces the higher dimensional case to discussing the Poincar\'e duals of the relevant lines in $H^{n-1}(X,\partial X;\Theta_{n-1})$. When $n=4$, this same discussion applies to the classification of smoothings up to stable isotopy (isotopy after product with $\R$) by the obstruction group $H^3(X,\partial X;\Z/2)$, but one typically encounters uncountably many isotopy classes (and diffeomorphism types) within each stable isotopy class. Note that the above method can be used to study sums of more general \textsc{cat} manifolds along collections of lines. In dimension 4, one can also consider actions on $\Sm(X)$ of the monoid $\Q_k$ of oriented smooth manifolds homeomorphic to a $k$-punctured 4-sphere $\Sigma_k$ with an order on the ends, generalizing the cases $\Q_1=\r$ and $\Q_2=\Q$ considered above. (The monoid operation is summing along $k$-fold unions of rays with a common endpoint; see the end of \cite{mod} for a brief discussion.) However, little is known about this monoid beyond what can be deduced from Corollaries~\ref{R4} and \ref{SxR} and the structure of $\r$ and $\Q$. It follows formally from having infinite sums that $\r$ has no nontrivial invertible elements, and no nontrivial homomorphism to a group \cite{infR4}; cf.\ also Theorem~\ref{hut}. However, the other monoids do not allow infinite sums. This leads to the following reformulation of Question~\ref{inverses}:

\begin{ques}\label{inverses2} Does $\Q$ (or more generally any $\Q_k$, $k\ge2$) have any nontrivial invertible elements? Is $H^3(\Sigma_k;\Z/2)$ the largest possible image of $\Q_k$ under a homomorphism to a group?
\end{ques}

\section{1-handle slides and 0/1-handle cancellation at infinity}\label{slides}

Our uniqueness result for adding $1$-handles at infinity (Theorem~\ref{main}) easily extends to adding both $0$- and $1$-handles at infinity, while allowing infinite slides and cancellation (Theorem~\ref{maincancel}).
With {\em compact} handles of index $0$ and $1$, one may easily construct countable handlebodies that are contractible, but are distinguished by their numbers of ends.
In this regard, adding $0$- and $1$-handles at infinity turns out to be simpler.
For instance, in each dimension at least four, every (at most) countable, connected, and oriented union of $0$- and $1$-handles at infinity is determined by its first Betti number.
As an application of Theorem~\ref{maincancel}, we give a very natural and partly novel proof of the hyperplane unknotting theorem.
The novelty here is that $0$- and $1$-handles at infinity provide the basic framework in which we employ Mazur's infinite swindle.\\

For simplicity, we assume throughout this section that all manifolds are oriented and all handle additions respect orientations.\\

Let $X$ be a possibly disconnected \textsc{cat} $n$-manifold where $n\geq4$.
Add to $X$ a collection of $0$-handles at infinity
$W= \bigsqcup_{i\in J} w_i$ where each $w_i$ is \textsc{cat} homeomorphic to $\R^n$.
The index set $J$ and all others below are discrete and countable.
Attach to $X\sqcup W$ a collection of $1$-handles at infinity
$H=\bigsqcup_{i\in S} h_i$ where each $h_i$ is \textsc{cat} homeomorphic to $[0,1]\times\R^{n-1}$ (see Figure~\ref{fig:add_handles}).
By Definition~\ref{onehandles} and Theorem~\ref{main}, $H$ is determined by multiray data
\begin{figure}[htbp!]
    \centerline{\includegraphics[scale=1.0]{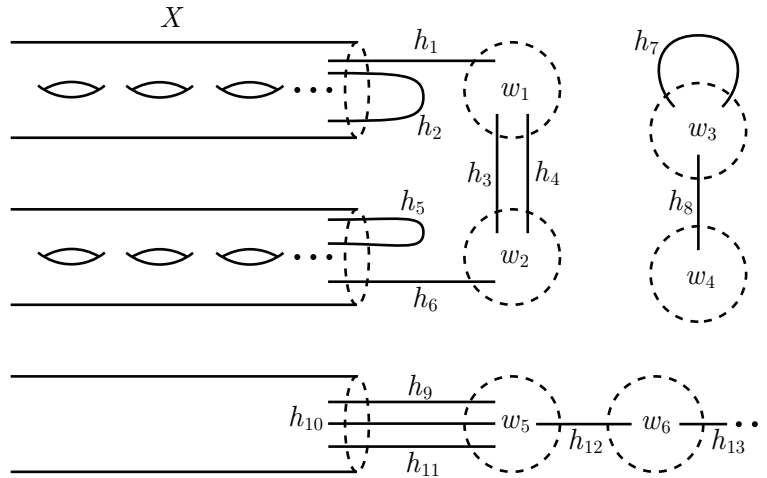}}
    \caption{Manifold $Z$ obtained from the manifold $X$ by adding $0$- and $1$-handles at infinity, the latter denoted by arcs.}
\label{fig:add_handles}
\end{figure}
$\gamma^-,\gamma^+\co S\times [0,\infty)\emb X\sqcup W$ with disjoint images.\\

To this data, we associate a graph $G$ defined as follows (see Figure~\ref{fig:graph}).
\begin{figure}[htbp!]
    \centerline{\includegraphics[scale=1.0]{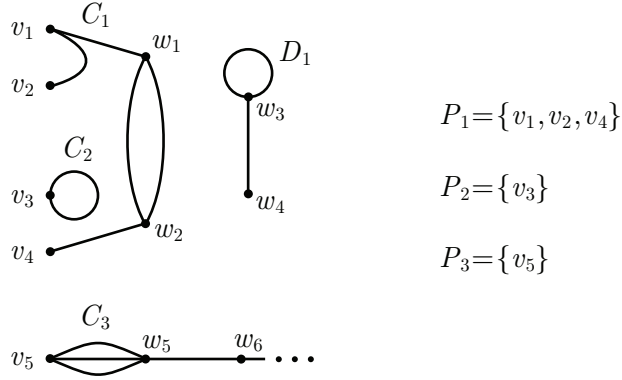}}
    \caption{Graph $G$ associated to the construction in Figure~\ref{fig:add_handles}, and induced partition of the vertices $v_i$ in $X$.}
\label{fig:graph}
\end{figure}
Let $\left\{v_i \mid i \in I \right\}$ be the set of proper homotopy classes of rays in the multiray data for $H$ that lie in $X$.
Each $v_i$ has at least one representative of the form $\gamma^-(j_i)$ or $\gamma^+(j_i)$ for some $j_i \in S$.
The vertex set $V$ of $G$ is:
\[
V:=\left\{v_i \mid i \in I \right\} \sqcup \left\{w_i \mid i\in J \right\}.
\]
The collection $E$ of edges of $G$ is bijective with the $1$-handles at infinity $H$ and thus is indexed by $S$.
The edge $e_i$, $i\in S$, corresponding to $h_i$ is formally defined to be the multiset of the two vertices in $V$
determined by the multiray data of $h_i$.
In particular, $E$ itself is a multiset, and the graph $G$ is countable, but is not necessarily locally finite, connected, or simple.
Indeed, $G$ may have multiple edges and loops.
Let $C=\bigsqcup_{i\in I(C)} C_i$ be the connected components of $G$ such that each component $C_i$ contains a vertex $v_{j(i)}$ in $X$.
Let $D=\bigsqcup_{i\in I(D)} D_i$ be the remaining components of $G$ where each component $D_i$ contains no vertex $v_j$ in $X$.
Notice that $C$ induces a partition $\mathcal{P}=\left\{P_j \mid j\in I(C)\right\}$ of $\left\{v_i \mid i \in I \right\}$
where $P_j$ is the subset of vertices in $\left\{v_i \mid i \in I \right\}$ that lie in $C_j$.
Below, Betti numbers $b_k$ are finite or countably infinite.

\begin{thm}\label{maincancel} For a \textsc{cat} $n$-manifold $X$ with $n\ge4$, the \textsc{cat} oriented homeomorphism type of the manifold $Z$ obtained by adding 0- and 1-handles at infinity to $X$ as above is determined by:
\begin{itemize}
\item[(a)] The set of pairs $\left(P_j,b_1\left(C_j\right)\right)$ where $P_j \in \mathcal{P}$.
\item[(b)] The multiset with elements $b_1(D_i)$ where $i\in I(D)$.
\end{itemize}
\end{thm}

Thus, we only need to keep track of which proper homotopy classes of rays in $X$ are used by at least one 1-handle (encoded as the vertices in each $P_j$), together with the most basic combinatorial data of the new handles. When the relevant ends are Mittag-Leffler, we can replace the ray data by the set of corresponding ends. The theorem implies that all 0-handles at infinity can be canceled except for one in each component of $Z$ disjoint from $X$, and that we can slide 1-handles over each other whenever their attaching rays are properly homotopic (e.g., whenever they determine the same Mittag-Leffler end). Furthermore, any reasonable notion of infinitely iterated handle sliding is allowed.

\begin{proof}
First, consider a component $D_i$ of $G$.
Let $M$ denote the component of $Z$ corresponding to $D_i$.
By Corollary~\ref{maincor}, we can and do assume that the rays used to attach $1$-handles at infinity in $M$ are radial (while still remaining proper and disjoint).
Then when $D_i$ is a tree, we can easily describe $M$ as a nested union of smooth $n$-disks, so it is a copy of $\R^n$.
In general, a spanning tree $T$ of $D_i$ determines a copy of $\R^n$ in $M$ (namely, one ignores a subset of the $1$-handles at infinity).
Thus, $M$ is $\R^n$ with $b_1(D_i)$ $1$-handles at infinity attached.
By Corollary~\ref{maincor}, such a manifold is determined by $b_1(D_i)$.\\

Second, consider a component $C_j$ of $G$.
Let $N$ denote the component of $Z$ corresponding to $C_j$.
Let $N'$ be the $n$-manifold obtained from $N$ as follows.
For each vertex $v_k$ in $C_j$, introduce a $0/1$-handle pair at infinity where the new $1$-handle
at infinity attaches to a ray in the class $v_k$ and to a ray in the new $0$-handle at infinity.
Also, the $1$-handles at infinity in $N$ attached to rays in the class of $v_k$ attach in $N'$ 
to rays in the new $0$-handle at infinity.
Theorem~\ref{main} implies that $N$ and $N'$ are \textsc{cat} oriented homeomorphic.
The graph $C'_j$ corresponding to $N'$ is obtained from $C_j$ by adding a leaf to each $v_k$.
Let $T$ be a spanning tree of the connected graph obtained by removing the new leaves from $C'_j$.
Then, $T$ determines a copy of $\R^n$ in $N'$.
This exhibits $N'$ as: the components of $X$ containing the vertices in $P_j$, a single $0$-handle at infinity $w_0$,
$b_1(C_j)$ oriented $1$-handles at infinity attached to $w_0$, and an oriented $1$-handle at infinity from each $v_k \in P_j$ to $w_0$.
\end{proof}

As an application of $1$-handle slides and $0/1$-handle cancellation at infinity, we prove the hyperplane unknotting theorem of Cantrell~\cite{C63} and Stallings~\cite{Stall65}. Recall that we assume \textsc{cat} embeddings are locally flat.

\begin{thm}\label{hut}
Let $f\co \R^{n-1}\to\R^n$ be a proper \textsc{cat} embedding where $n\geq4$, and let $H=f\left(\R^{n-1}\right)$.
Then, there is a \textsc{cat} homeomorphism of $\R^n$ that carries $H$ to a linear hyperplane.
\end{thm}

A \textsc{cat} ray in $\R^k$ is \emph{unknotted} provided there is a \textsc{cat} homeomorphism of $\R^k$ that carries the ray to a linear ray. Recall that each \textsc{cat} ray in $\R^k$, $k\geq4$, is unknotted. For \textsc{cat}=\textsc{pl} and \textsc{cat}=\textsc{diff}, this fact follows from general position, but for \textsc{cat}=\textsc{top} it is nontrivial and requires Homma's method (see Lemma~\ref{ambient} above and~\cite[\S~7]{CKS}). Thus, the following holds under the hypotheses of Theorem~\ref{hut} by taking $r$ to be the image under $f$ of a linear ray in $\R^{n-1}$: \emph{There is a \textsc{cat} ray $r\subset H$ that is unknotted in both $H$ and $\R^n$, where the former means $f^{-1}(r)$ is unknotted in $\R^{n-1}$.}

The hyperplane $H$ separates $\R^n$ into two connected components by Alexander duality.
Let $A'$ and $B'$ denote the closures in $\R^n$ of these two components as in Figure~\ref{fig:A_and_B}.
\begin{figure}[htbp!]
    \centerline{\includegraphics[scale=1.0]{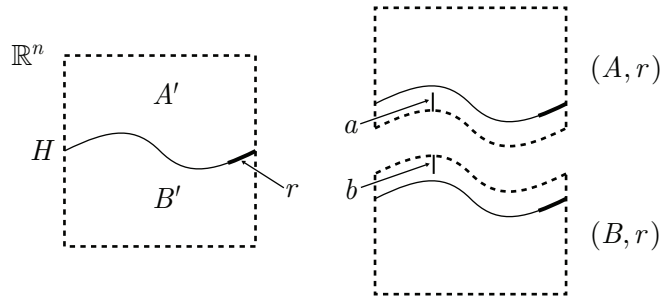}}
    \caption{Closures $A'$ and $B'$ of the complement of $H$ in $\R^n$ (left) and their unions $A$ and $B$ with open collars on $H$ (right).}
\label{fig:A_and_B}
\end{figure}
So, $\partial A'=H=\partial B'$, and $H$ has a bicollar neighborhood in $\R^n$.
Using the bicollar, define:
\begin{align*}
A:=&A'\cup(\textnormal{open collar on $H$ in $B'$})\\
B:=&B'\cup(\textnormal{open collar on $H$ in $A'$})
\end{align*}
as in Figure~\ref{fig:A_and_B}.
Figure~\ref{fig:A_and_B} also depicts \textsc{cat} rays $a\subset A$ and $b\subset B$ that are radial with respect to the collarings.
Evidently, $a$ and $b$ are \textsc{cat} ambient isotopic to $r$ in $A$ and $B$ respectively. (These simple isotopies have support in a neighborhood of the open collars).\\

\begin{lem}\label{ABcuhs}
It suffices to show that $A'$ and $B'$ are \textsc{cat} homeomorphic to closed upper half-space $\R^n_+$.
\end{lem}

\begin{proof}
We are given \textsc{cat} homeomorphisms $g\co A'\to\R^n_+$ and $h\co B'\to\R^n_+$.
Replace $h$ by its composition with a reflection so that $h$ maps $B'\to\R^n_-$.
Note that $g$ and $h$ need not agree pointwise on $H$.
Identify $\R^{n-1}\times\left\{0\right\}$ with $\R^{n-1}$.
We have a \textsc{cat} homeomorphism $j\co \R^{n-1}\to\R^{n-1}$ given by the restriction of $g\circ h^{-1}$ to $\R^{n-1}$.
Define the \textsc{cat} homeomorphism $k\co B'\to\R^n_-$ by $k=(j\times\textnormal{id})\circ h$ (that is, compose $h$ with $j$ at each height).
Now, $g$ and $k$ agree pointwise on $H$. For \textsc{cat}=\textsc{top} and \textsc{cat}=\textsc{pl}, the proof of the lemma is complete.
For \textsc{cat}=\textsc{diff}, one smooths along collars as in Hirsch~\cite[Theorem~1.9, p.~182]{H94}.
\end{proof}

We will use the symbols in Figure~\ref{fig:hut_notation} to denote the indicated manifold/ray pairs.
Here, $c$ is a radial ray in $\R^n$.
\begin{figure}[htbp!]
    \centerline{\includegraphics[scale=1.0]{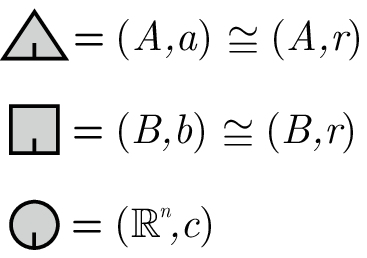}}
    \caption{Notation for relevant manifold/ray pairs.}
\label{fig:hut_notation}
\end{figure}
All rays in this proof, such as $a$ and $b$, will be parallel (\textsc{cat} ambient isotopic) to $r$ or $c$.
An added $1$-handle at infinity will be denoted by an arc connecting such symbols as in Figure~\ref{fig:hut_cancel}.\\\

\begin{lem}\label{ABcancel}
All three of the manifold/ray pairs in Figure~\ref{fig:hut_cancel} are \textsc{cat} homeomorphic to one another.
\begin{figure}[htbp!]
    \centerline{\includegraphics[scale=1.0]{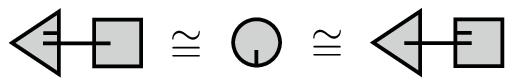}}
    \caption{Isomorphic manifold/ray pairs.}
\label{fig:hut_cancel}
\end{figure}
\end{lem}

\begin{proof}
First, we claim that adding a $1$-handle at infinity to $(A,a)\sqcup(B,b)$ yields $\R^n$.
Recalling the collars in Figure~\ref{fig:A_and_B}, the claim would be evident if we could choose the tubular neighborhood maps for the $1$-handle at infinity to be the full collars in the $\R^{n-1}$ directions. However, an open tubular neighborhood must, by our definition, extend to a closed tubular neighborhood. So, instead we use smaller tubular neighborhoods inside the collars as follows. Identify the collar on $H$ in $A$ with $\R^{n-1}\times [0,1)$ so that $H$ corresponds to $\R^{n-1}\times\left\{0\right\}$ and the ray $a$ corresponds to $\left\{0\right\}\times [1/2,1)$. For each $t\in[1/2,1)$, there is an open horizontal $(n-1)$-disk in $\R^{n-1}\times[0,1)$ at height $t$, of radius $1/(1-t)$, and with center on $a$. The union of these disks is our desired open tubular neighborhood of $a$. Similarly, we obtain an open tubular neighborhood of $b$ using the compatible collar in $B$. The claim follows by attaching the $1$-handle at infinity using these tubular neighborhood maps and reparameterizing collars.
Next, let $a'$ and $b'$ be the indicated rays in Figure~\ref{fig:hut_cancel} parallel to $a$ and $b$ respectively.
The lemma follows by shrinking the above tubular neighborhood maps in the $\R^{n-1}$ directions to be disjoint from $a'$ and $b'$ respectively.
\end{proof}

\begin{lem}\label{pairs_std}
It suffices to prove that $(A,a)$ and $(B,b)$ are \textsc{cat} homeomorphic as pairs to $(\R^n,c)$.
\end{lem}

\begin{proof}
First, consider the cases \textsc{cat}=\textsc{diff} and \textsc{cat}=\textsc{pl}.
The collar on $H$ in $A$ is a \textsc{cat} \emph{closed} regular neighborhood of $a$ in $A$ with boundary $H$.
Using the hypothesis $(A,a)\cong (\R^n,c)$, apply uniqueness of such neighborhoods in $(\R^n,c)$ to see that $A'$ is \textsc{cat} homeomorphic to $\R^n_+$.
Similarly, $B'$ is \textsc{cat} homeomorphic to $\R^n_+$.
Now, apply Lemma~\ref{ABcuhs}.\\

For \textsc{cat}=\textsc{top}, we are given a homeomorphism $g\co (A,a)\to(\R^n,c)$.
Let $V\cong\R^n_+$ be the collar added to $A'$ along $H$ to obtain $A$ as in Figure~\ref{fig:A_and_B}.
Let $U\cong\R^n_+$ be a collar on $H$ in $A$ on the opposite side of $H$ as in Figure~\ref{fig:mcn}.
\begin{figure}[htbp!]
    \centerline{\includegraphics[scale=1.0]{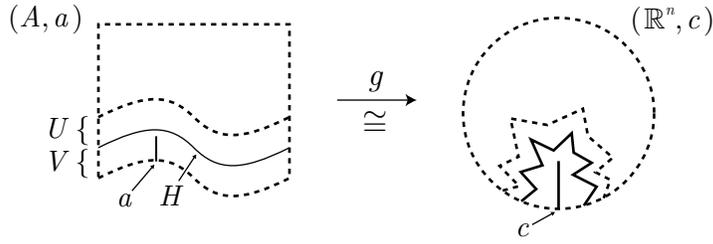}}
    \caption{Homeomorphic manifold/ray pairs $(A,a)$ and $(\R^n,c)$. Also depicted are the hyperplane $H$, the collar $V$ added to $A'$ to obtain $A$, a collar $U$ on the other side of $H$, and their images in $\R^n$.}
\label{fig:mcn}
\end{figure}
Recall that $\R^n$ itself is an open mapping cylinder neighborhood of $c$ in $\R^n$ (see~\cite{KR} and~\cite[pp.~1816,1831]{CKS}).
Similarly, $U\cup V$ is an open mapping cylinder neighborhood of $a$ in $U\cup V$.
So, $g(U\cup V)$ is another open mapping cylinder neighborhood of $c$ in $\R^n$.
Uniqueness of such neighborhoods (see~\cite{KR} and~\cite{CKS}) implies there exists a homeomorphism $h\co g(U\cup V)\to\R^n$ that fixes $g(V)$ pointwise.
Therefore:
\[
g(U)\cong \R^n-\textnormal{Int}g(V)=g(A').
\]
Hence, $A'\cong U \cong \R^n_+$. Similarly, $B'$ is homeomorphic to $\R^n_+$.
Again, Lemma~\ref{ABcuhs} completes the proof.
\end{proof}

\begin{figure}[htbp!]
    \centerline{\includegraphics[scale=1.0]{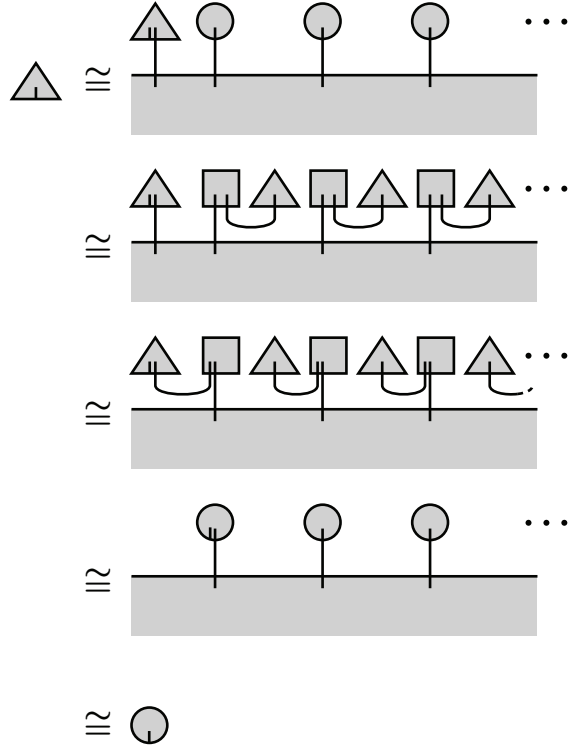}}
    \caption{Mazur's infinite swindle as $1$-handle slides and $0/1$-handle cancellations at infinity.}
\label{fig:hut_pf}
\end{figure}

Finally, we come to the heart of the proof of the hyperplane unknotting theorem.
Mazur's infinite swindle~\cite{M59} is realized as $1$-handle slides and $0/1$-handle cancellations at infinity.
Figure~\ref{fig:hut_pf} proves that $(A,a)$ is \textsc{cat} homeomorphic to $(\R^n,c)$.
In Figure~\ref{fig:hut_pf}, the horizontal region is a copy of $\R^n$.
The first, third, and fifth isomorphisms in Figure~\ref{fig:hut_pf} hold by Theorem~\ref{maincancel}.
The second and fourth isomorphisms hold by Lemma~\ref{ABcancel}.
With $(A,a)\cong (\R^n,c)$, Figure~\ref{fig:hut_cancel} implies that $(B,b)\cong (\R^n,c)$.
By Lemma~\ref{pairs_std}, our proof of the hyperplane unknotting theorem is complete.\\


\end{document}